\documentclass[a4paper]{article}
\usepackage[english]{babel}
\usepackage[a4paper, top=3cm, bottom=3cm, left=3cm, right=3cm]{geometry}
\usepackage[T1]{fontenc}
\usepackage[utf8]{inputenc}
\usepackage{amsmath} 
\usepackage{amsfonts} 
\usepackage{amssymb}
\usepackage{amsthm}
\usepackage{graphicx}
\usepackage{fancyhdr}
\usepackage{xcolor}

\usepackage{pgfplotstable}
\usepackage{booktabs}
\pgfplotsset{compat=1.18}

\usepackage{url}

\theoremstyle{plain}

\newtheorem{thm}{Theorem}
\newtheorem{defn}{Definition}
\newtheorem{rem}{Remark}

\title {Numerical entropy production in finite volume $P_0P_M$ ADER schemes}
\author{M. Semplice%
\thanks{Università degli Studi dell'Insubria - Dipartimento di Scienza e Alta Tecnologia - Como (Italy) {\sl matteo.semplice@uninsubria.it}}
, A. Zappa%
\thanks{Università degli Studi dell'Insubria - Dipartimento di Scienze Teoriche ed Applicate - Como (Italy) {\sl azappa1@uninsubria.it}}}
\date{}
\newcommand{\dx}{\,dx}
\newcommand{\dt}{\,dt}

\usepackage{todonotes}
\let\revUno\relax
\let\revDue\relax

\graphicspath{{Immagini/}}
\begin{document}
\bibliographystyle{alpha}

\maketitle

\section*{Abstract}
{\revDue We consider the numerical integration of conservation laws endowed with an entropy inequality and we study the residual of the scheme on this inequality, which represents the numerical entropy production.} This idea has been introduced and exploited in Runge-Kutta finite volume methods, where the numerical entropy production has been used as an indicator in adaptive schemes, since it scales as the local truncation error of the method for smooth solutions and it highlights the presence of discontinuities and their kind.
\\
The aim of this work is to extend this idea to finite volume $P_0P_M$ ADER timestepping techniques. We show that the numerical entropy production can be defined also in this context and it provides a scalar quantity computable for each space-time volume which, under grid refinement, decays to zero with the same rate of convergence of the scheme for smooth solutions. Its size gradually increases when the local solution regularity lowers, remaining bounded up to contact discontinuities and divergent on shock waves. 
Theoretical results are proven in a multi-dimensional setting on {\revUno arbitrary polygonal grids}. We also present numerical evidence showing that it is essentially negative definite. Moreover, we propose an example of $p$-adaptive scheme that uses the numerical entropy production as a-posteriori smoothness indicator. The scheme locally modifies its order of convergence with the purpose of removing the oscillations due to the high-order of accuracy of the scheme.

\paragraph{Keywords} 
Finite volume $P_0P_M$ schemes, 
hyperbolic systems, multiple space dimensions,
numerical entropy, p-adaptivity
\section{Introduction}

We consider systems of conservation laws of the form
\begin{equation}
    \partial_tu+\nabla_x\cdot f(u)=0 \label{conslaw}
\end{equation}
 where the solutions $u:\mathbb{R}^+\times{\revDue \mathbb{R}^{\nu}}\rightarrow\mathbb{R}^m$ are functions of time $t$ and a spatial variable $x\in{\revDue \mathbb{R}^{\nu}}$ and the flux is a function $f:\mathbb{R}^m\rightarrow\mathbb{R}^m$. We assume that the system is hyperbolic, which means that there exists a complete system of real eigenvectors of the Jacobian matrix of the  flux in direction $\vec{n}$ for every possible physical state $u$ and $\vec{n}\in {\revDue S^{\nu-1}}$.

 When developing highly efficient numerical schemes, one has to consider the possible presence in the solution of steep gradients or discontinuities that may develop in finite time even from smooth initial data.  In the vicinity of such zones, the real accuracy of formally high order schemes is reduced and spurious oscillations may not be effectively controlled by the Essentially Non-Oscillatory reconstructions employed. On the other hand, in smooth areas of the solutions, it is more cost-efficient to apply an high-order scheme on a coarsen mesh than a first-order one on an extremely refined grid. One is thus lead to develop adaptive numerical schemes.

 Adaptive schemes have the ability to adapt locally (in space and time) to the solution being computed by (any combination of) changing the mesh size refining or coarsening the computational mesh ($h$-refinement or AMR), redistributing the cells without changing their global number ($r$-refinement), changing the order of the applied scheme ($p$-refinement, MOOD), or changing the features of the scheme (e.g. applying the local characteristic decomposition).
 
 A common requirement of any adaptive procedure is the need of an error indicator to drive the adaption. To this end, in the context of finite volume schemes for hyperbolic conservation laws, the concept of numerical entropy production has been introduced in \cite{Puppo04:numerical:entropy} and extended to high order \cite{PS11:numerical:entropy} and balance laws \cite{PS16:entropy:balance}.
The numerical entropy production has been used to drive adaptive schemes in the AMR \cite{Ersoy:13,Abbate:19,SCR:cwenoAMR}, 
in the $p$-adaptive \cite{PS11:numerical:entropy} and in the MOOD framework \cite{SL:18:AMRMOOD}.
Furthermore, \cite{2019:Lozano:entAdj} exploits the numerical entropy production in steady state computations to drive grid adaptation when the quantity of interest is the aerodynamics drag of an object immersed in the computational domain.
Finally, the numerical entropy production has also been exploited for uses other then grid adaptivity. For example, in \cite{Guermond:11} an efficient entropy viscosity method for nonlinear conservation laws is proposed, in which the scheme is locally modified adding numerical diffusion proportionally to the local production of entropy.

 The employment of the numerical entropy production is natural for several reasons. Firstly, it is possible that solutions for a system of conservation laws may lose their regularity in a finite time even if the initial data are smooth, giving rise to discontinuities. In this case, the existence of strong and the uniqueness of weak solutions are lost. Admissible weak solutions should however satisfy the entropy inequality
 \begin{equation}
     \partial_t\eta(u)+\nabla_x\cdot \psi(u)\leq0 \label{entrin}
 \end{equation}
 where $\eta$ is a scalar convex entropy function and $\psi$ is the corresponding entropy flux such that it satisfies the compatibility condition $\nabla^T\eta J_f(u)=\nabla^T\psi$.
 For exact solutions $\eqref{entrin}$ holds as an equality unless a shock wave is present in the solution. For systems of conservation laws describing physical systems, one can usually employ the physical entropy.

 The numerical entropy production $S_j^n$ is a residual of $\eqref{entrin}$ computed on the numerical solutions of $\eqref{conslaw}$ and introducing numerical entropy fluxes $\Psi$ which are consistent with the exact entropy flux $\psi$. For high order finite volume schemes it has been shown that $S_j^n$ converges to zero with the same rate of the local truncation error for smooth solutions and it is bounded by terms of order $\mathcal{O}\left(1/\Delta t\right)$ if the solution presents a shock wave. 
 For this reasons, $S_j^n$ can be considered as an a-posteriori error indicator for the finite volume scheme.
 
 There are also some theoretical studies about the sign of the numerical entropy production. In \cite{PS11:numerical:entropy}, it has been proved that it is essentially negative definite for a first-order Rusanov scheme applied to a scalar conservation law. For high-order schemes, in depth studies on the sign have been conducted in \cite{Lozano:18,2019:Lozano:entRKImpl}, where it is shown that in general a high order Runge-Kutta time discretization will produce spurious numerical entropy and that entropy stability of the resulting scheme can be obtained under a CFL restriction also in the implicit case. We point out that in our definition of the numerical entropy production, we do not need to use a specific numerical entropy flux, but any numerical flux compatible with $\psi$ can be employed.

 The main idea of this work is to extend the definition of $S_j^n$ given in \cite{PS11:numerical:entropy} to Finite Volume Arbitrary Accuracy DERivative Riemann problem (FV-ADER) schemes, which are fully-discrete schemes that solve the high-order Riemann problem approximately without semidiscretization nor Runge-Kutta methods. More precisely, the approach of \cite{Dumbser:2008:PnPm} is followed, which is more flexible than previous ones based on an analytic or semi-analytic version of the Cauchy–Kovalewskaya or Lax–Wendroff procedure. We point out that ADER schemes can be reinterpreted in the deferred correction approach (DeC) {\revUno \cite{2021:DeC,2024:ADERDeC}} and therefore as Runge-Kutta methods. Nevertheless,  it seems not practical to use the corresponding Butcher tableaux directly for the computation of the numerical entropy production.
 
 In general, ADER schemes of \cite{Dumbser:2008:PnPm} use two sets of piecewise polynomials with degree $N$ and $M$. The first set represents the spatial solution in each element at each time step, whereas the second set contains space-time polynomials used for the evolution in time of the solution. These latter are initialized at time $t^n$ via a reconstruction procedure and evolved in time through a local continuous space-time Galerkin method \cite{Dumbser:2008:PnPm}, giving a high-order predictor for the conserved variables, which is then used to compute the $N$-th degree spatial polynomial at time $t^{n+1}$ in a conservative manner. For this reason, they are also named $P_NP_M$ schemes. We are interested in the $P_0P_M$ approach, which corresponds to finite volume-type schemes of order $M+1$.
 
 The previous results of \cite{PS11:numerical:entropy} are here extended in two ways: first, the semidiscrete approach is replaced with a fully-discrete ADER time advancement and, second, the theoretical decay of the numerical entropy production is studied in multidimensional setting on arbitrary meshes. In this paper a simple example of the application of the numerical entropy production is given in a $p$-adaptive scheme, in which $S_j^n$ is used as a-posteriori error indicator.

 The paper is organized as follows. In Section 2 we recall the construction of semidiscrete finite volume schemes evolved in time with a Runge-Kutta method. We also recall the definition of the numerical entropy production and its properties. In Section 3 we introduce the FV-ADER $P_0P_M$ schemes and we extend the definition of the numerical entropy production to this setting. In Section 4 we test numerically the properties of $S_j^n$ in the one-dimensional and two-dimensional case. We show numerical evidence of the essentially negative sign of $S_j^n$ and its different behaviour with respect to the kind of wave we are dealing with. Finally, in Section 5 we present an example of $p$-adaptive scheme that exploits the numerical entropy production as indicator to decrease the order of the scheme in the regions in which one wants to avoid spurious oscillations.

\section{Numerical Entropy for Semidiscrete Runge-Kutta schemes}

We consider conservation laws of the form
\begin{equation}
	\partial_tu+\nabla_x\cdot f(u)=0  \label{cl}
\end{equation}
coupled with the entropy inequality
\begin{equation}
    \partial_t\eta(u)+\nabla_x\cdot \psi(u)\leq0
\end{equation}
over the domain $\Omega = \bigcup_{j\in\mathbb{Z}}\Omega_j$ , where $\Omega_j$ are cells partitioning the computational domain $\Omega$.
Dividing \eqref{cl} by the cell volume $|\Omega_j|$ and integrating over $\Omega_j$,  we obtain
\begin{equation*}
	\dfrac{d}{dt}\overline{u}_j(t) +\dfrac{1}{|\Omega_j|}\int_{\partial\Omega_j}f(u(t,x))\cdot\hat{n}\dx=0
\end{equation*}
where 
\begin{equation*}
    \overline{u}_j(t)=\dfrac{1}{|\Omega_j|}\int_{\Omega_j}u(t,x)\dx 
\end{equation*}
is the exact cell average in $\Omega_j$ and $\hat{n}$ the outgoing normal. 
Integrating also in time, we get
\begin{equation}
	\overline{u}_j^{n+1} = \overline{u}_j^n-\dfrac{1}{|\Omega_j|}\int_{t^n}^{t^{n+1}}\int_{\partial\Omega_j}f(u(\tau,x))\cdot\hat{n}\dx d\tau. 
    \label{exafv}
\end{equation}
Similarly, the entropy inequality becomes
\begin{equation}
	\overline{\eta}_j^{n+1} - \overline{\eta}_j^n + \dfrac{1}{|\Omega_j|}\int_{t^n}^{t^{n+1}}\int_{\partial\Omega_j}\psi(u(\tau,x))\cdot\hat{n}\dx d\tau\leq0. \label{exaei}
\end{equation}
The space integrals on $\partial\Omega_j$ can be replaced with suitable quadrature rules.

If the solution is smooth, inequality $\eqref{exaei}$ is actually an equality, whereas if it has a shock discontinuity, then the inequality holds strictly for the physical solution of the conservation law.

In the numerical solution, we approximate $\overline{u}_j^n$ with $\overline{U}_j^n$.  Since we are not able to compute $f(u(t,x_{j+1/2}))$ directly, we need a reconstruction, which allows to compute the values of $U$ at the boundary of the cells.
Let $R_j(t,x)$ be a local reconstruction, e.g. a polynomial of fixed maximum degree, that has $\int_{\Omega_j} R_j(t,x)\dx= |\Omega_j| \overline{U}_j(t) $ and which approximates accurately also the cell averages in the neighbors of $|\Omega_j|$. The operator $\mathcal{R}_j$ thus maps the cell averages in $\Omega_j$ and its neighbors to a reconstruction polynomial defined in $\Omega_j$. Then, for $x\in\partial\Omega_j$ we denote the inner and outer reconstructed values as
\begin{equation*}
	U^{-}(t,x)=R_j(t,x)
    \qquad 
	U^{+}(t,x)=R_k(t,x),
\end{equation*}
where $\Omega_k$ is the neighbor of $\Omega_j$ such that $x\in\Omega_j \cap \Omega_k$.
The outer value is well defined except at the corners of the cells, so it is computable at least for any Gaussian quadrature point on the edges of the cells in the mesh.

Then $f(u(t,x))$ in $\eqref{exafv}$ can be replaced by $\mathcal{F}(U^-(t,x),U^+(t,x))$, where $\mathcal{F}$ is a two-point numerical flux consistent with $f(u)$. Therefore, the numerical scheme becomes
\begin{equation}
	\overline{U}_j^{n+1} = \overline{U}_j^n -\dfrac{1}{|\Omega_j|}\int_{t^n}^{t^{n+1}}\int_{\partial\Omega_j}\mathcal{F}(U^-(\tau,x),U^+(\tau,x))\cdot\hat{n}\dx d\tau.
\end{equation}

Following \cite{PS11:numerical:entropy}, the numerical entropy production $S_j^n$ on each cell $\Omega_j$ is defined as the residual of the entropy inequality \eqref{exaei}, that is
\begin{equation}
    S_j^n = \dfrac{1}{\Delta t}\left( \overline{\eta(U^{n+1})}_j-\overline{\eta(U^n)}_j+\dfrac{1}{|\Omega_j|}\int_{t^n}^{t^{n+1}}\int_{\partial\Omega_j}\Psi(U^-(\tau,x),U^+(\tau,x))\cdot\hat{n}\dx d\tau\right)
\end{equation}
where $\Psi$ is a numerical flux consistent with the exact entropy flux $\psi(u)$.

If we apply a $p^{th}$-order Runge-Kutta method to evolve in time the semidiscrete scheme we get
\begin{equation}
    \overline{U}_{j}^{n+1} = \overline{U}_j^n - \Delta t\sum_{i=1}^{{\revDue \sigma}}b_iK_j^{(i)}
    \label{semidisc}
\end{equation}
where $K_j^{(i)}=\dfrac{1}{|\Omega_j|}\int_{\partial\Omega_j}\mathcal{F}(U^{(i),-},U^{(i),+})\cdot\hat{n}\dx$ is {\revDue the integral of the numerical flux evaluated on} the $i^{th}$ Runge-Kutta stage, whose cell averages are
\begin{equation*}
    \overline{U}_j^{(i)}=\overline{U}_j^n-\Delta t\sum_{k=1}^{i-1}a_{ik}K_j^{(k)}
\end{equation*}
and $a_{ik}$ and $b_i$ are the coefficients from the Butcher's tableaux $(A,b)$.

Therefore, for a finite volume Runge-Kutta scheme the numerical entropy production over the volume $[t^n,t^{n+1}]\times\Omega_j$ can be computed as
\begin{equation}
    S_j^n = \dfrac{1}{\Delta t}\left(\dfrac{1}{|\Omega_j|}\mathcal{Q}_{\Omega_j}(\eta(U^{n+1}))-\dfrac{1}{|\Omega_j|}\mathcal{Q}_{\Omega_j}(\eta(U^n))+\Delta t\sum_{i=1}^{{\revDue \sigma}}b_i\tilde{K}_j^{(i)}\right)\label{sfv}
\end{equation}
where $\tilde{K}_j^{(i)}=\dfrac{1}{|\Omega_j|}\int_{\partial\Omega_j}\Psi({U}^{(i),-},{U}^{(i),+})\cdot\hat{n}\dx$ and $\mathcal{Q}_{\Omega_j}$ is a $q^{th}$ order accurate quadrature rule in space.

\begin{rem}
 The computation of $S_j^n$ does not affect significantly the computational cost of the scheme. Indeed, most terms in $\eqref{sfv}$ are already available from previous computations and one just needs an extra reconstruction of the solution from the cell averages at time $t^{n+1}$, which can of course be reused at the beginning of the next timestep. 
 
 In order to have an even more efficient implementation, the entropy function $\eta$ can be added as an additional variable in the system of conservation laws, 
 reinitializing its value at the beginning of each time step as $\overline{\eta}_j^n=\eta(\overline{U}_j^n)$.
Applying the numerical scheme to the augmented system, one computes
\begin{equation*}
    \overline{\eta}_j^{n+1}=
    \overline{\eta}_j^n-\Delta t\sum_{i=1}^{{\revDue \sigma}}b_i\tilde{K}_j^{(i)}=
    \dfrac{1}{|\Omega_j|}\mathcal{Q}_{\Omega_j}(\eta(U^n))-\Delta t\sum_{i=1}^{{\revDue \sigma}}b_i\tilde{K}_j^{(i)}.
\end{equation*}
At this point, $S_j^n$ can be computed as
\begin{equation*}
    S_j^n=\dfrac{1}{\Delta t}\left(\dfrac{1}{|\Omega_j|}\mathcal{Q}_{\Omega_j}(\eta(U^{n+1}))-\overline{\eta}_j^{n+1}\right).
\end{equation*}
\end{rem}

In \cite{PS11:numerical:entropy} the following theorem is proved, for the one-dimensional case of a conservation law.
\begin{thm}\label{thm:FVRK}
Consider a $p^{th}$-order semidiscrete finite volume scheme of the form $\eqref{semidisc}$ and define the numerical entropy production as in $\eqref{sfv}$. Assume $\Delta t=\lambda\Delta x$, where $\lambda$ is the mesh ratio.
\begin{itemize}
	\item If the solution is regular, then the numerical entropy production in a fixed time step decays to zero as $\mathcal{O}(\Delta t^p)$ for $\Delta t\to0$.
	\item In cells in which the solution is not regular, $S_j^n\leq\dfrac{C}{\Delta t}$ for a constant $C$ which does not depend on $\Delta t$.
\end{itemize}
\end{thm}
We point out that although Theorem~\ref{thm:FVRK} was proven in the one-dimensional setting, the numerical entropy production has been exploited also in multi-dimensional  AMR schemes in \cite{SCR:cwenoAMR,SL:18:AMRMOOD}.

\section{Numerical Entropy Production by FV-ADER schemes}

The aim of this section is to extend the definition of the numerical entropy production to the case of FV-ADER $P_0P_M$ schemes \cite{Dumbser:2008:PnPm}. 
The idea is similar to \cite{PS11:numerical:entropy}, 
but the last term of \eqref{sfv} which arose from the Runge-Kutta stages of the scheme for time advancement, will need to be replaced with a term linked to the FV-ADER evolution algorithm.

We start again from the conservation law \eqref{cl}. Integrating over $[t^n,t^{n+1}]\times\Omega_j$ and dividing by $|\Omega_j|$, we get
\begin{equation*}
	\overline{u}_j^{n+1} =\overline{u}_j^n-\dfrac{1}{|\Omega_j|}\int_{t^n}^{t^{n+1}}\left(\int_{\partial\Omega_j}f(u(\tau,x))\cdot\hat{n}\dx\right)d\tau.
\end{equation*}

Using an appropriate quadrature rule in time $\mathcal{Q}_t$, we obtain
\begin{equation*}
	\overline{u}_j^{n+1} \simeq \overline{u}_j^n-\dfrac{1}{|\Omega_j|}\mathcal{Q}_t\left(\int_{\partial\Omega_j}f(u(t,x))\cdot\hat{n}\dx \right).
\end{equation*}    

We approximate $\overline{u}_j^n$ with the numerical solution $\overline{U}_j^n$ and the exact flux $f(u(t,x))$ with a 2-point numerical flux $\mathcal{F}(u^-(t,x),u^+(t,x))$ consistent with $f(u)$. Here $u^{\pm}$ denote the inner and outer values of $u$ at the cell boundary and they are further approximated by the following local ADER predictor: in each cell $\Omega_j$ we look for a solution of the conservation law of the form 
\begin{equation}
	\hat{U}_j(\tau,{\xi})=\sum_{l}\hat{u}_{j,l}\theta_l(\tau,{\xi})
    \label{pred}
\end{equation}
where $\xi\in{\revDue \mathbb{R}^{\nu}}$ and $\{\theta_l(\tau,{\xi})\}$ are the basis functions for a space-time FEM method of order $M+1$.

The degrees of freedom $\{\hat{u}_{j,l}\}$ are computed in each cell $\Omega_j$ through the local continuous space-time Galerkin method presented in \cite{Dumbser:2008:PnPm}. 
Performing the change of variables with the local time $\tau=(t-t^n)/\Delta t$ and the coordinate transformation from $\Omega_j$ to the reference element $\Omega_E$, \eqref{cl} can be rewritten in terms of the reference coordinates $(\tau,\xi)$. 
Multiplying \eqref{cl} by the space-time test functions $\theta_k(\tau,\xi)$ and by $\Delta t$ and integrating in space and time over the  reference element $[0,1]\times\Omega_E$ we obtain:
\begin{equation*}
	\int_{\Omega_E}\int_0^1\theta_k(\tau,\xi)\partial_{\tau}u(\tau,\xi)d\tau d\xi+\sum_{{\revDue r}=1}^{\revDue \nu}{\Delta t}\int_{\Omega_E}\int_0^1\theta_k(\tau,\xi)\partial_{\xi_{\revDue r}}{f}_{\revDue r}^{*}(u(\tau,\xi))d\tau d\xi = 0
\end{equation*}
where ${f}_{\revDue r}^{*}(u(\tau,\xi))=\sum_{s=1}^{\revDue \nu}f_{s}(u(\tau,\xi))\partial_{x_{s}}\xi_{\revDue r}$.
If we approximate the solution $u$ and the flux $f(u)$ using \eqref{pred} and we integrate the first integral by parts with respect to the time variable, we obtain
\begin{equation*}
\begin{aligned}
	\sum_l\hat{u}_l\left[-\int_{\Omega_E}\int_0^1\partial_{\tau}\theta_k(\tau,\xi)\theta_l(\tau,\xi)d\tau d\xi +\int_{\Omega_E}\theta_k(1,\xi)\theta_l(1,\xi)d\xi-\int_{\Omega_E}\theta_k(0,\xi)\theta_l(0,\xi)d\xi\right] \\
	+ \sum_{{\revDue r}=1}^{\revDue \nu}{\Delta t}\sum_l\hat{f}_{l,{\revDue r}}^{*}\int_{\Omega_E}\int_0^1\theta_k(\tau,\xi)\partial_{\xi_{\revDue r}}\theta_l(\tau,\xi)d\tau d\xi =0
\end{aligned}
\end{equation*}
that can be written in the {\revDue matrix form}
\begin{equation*}
	-K^{\tau}\hat{u}+F^1\hat{u}-F^0\hat{u}+\sum_{{\revDue r}=1}^{\revDue \nu}{\Delta t}K^{\xi_{\revDue r}}\hat{f}^*=0
\end{equation*}
where $K^{\tau}$, $K^{\xi_{\revDue r}}$, $F^1$ and $F^0$ are respectively the matrices of the time derivative, the space derivative and the space integrals at $t=1$ and $t=0$.

Rearranging, we obtain an iterative method to find $\hat{u}$:
\begin{equation*}
	\hat{u}_{l}^{(k+1)}=(-K^{\tau}+F^1)^{-1}\left(F^0\hat{u}_{l}^{(0)}-\sum_{{\revDue r=1}}^{{\revDue \nu}}{\Delta t}K^{\xi_{\revDue r}}\hat{f}_{l,{\revDue r}}^{*,(k)}\right) \label{eq:iter}
\end{equation*}
where $\hat{f}_{l,{\revDue r}}^*=f_{\revDue r}^*(\hat{u}_{l})$. The degrees of freedom at the first step are initialized as the reconstructed values ${R}_j(t^n,x)$ at time $t^n$.

In \cite{Dumbser:2008:PnPm} it has been observed numerically that for nonlinear systems the method converges to a precision of $10^{-9}$ with at most $M$ or $M+1$ iterations, where $M+1$ is the order of convergence of the scheme. This is further confirmed by the analogy between ADER and DeC schemes studied in \cite{2024:ADERDeC}. 

The final numerical scheme becomes
\begin{equation}
    \overline{U}_j^{n+1} = \overline{U}_j^n-\dfrac{1}{|\Omega_j|}\mathcal{Q}_t\left(\int_{\partial\Omega_j}\mathcal{F}(\hat{U}^-(t,x),\hat{U}^+(t,x))\cdot\hat{n}\dx\right) \label{fv5}
\end{equation}
where $\hat{U}^{\pm}$ denote the inner and outer values of the predictor $\hat{U}$ at the cell boundaries and the integral on $\partial\Omega_j$ is computed using an appropriate quadrature rule in space. 

\begin{defn}
The numerical entropy production by a FV-ADER $P_0P_M$ scheme is obtained by the formula
\begin{equation}
	S_j^n=\dfrac{1}{\Delta t|\Omega_j|}\left\{\mathcal{Q}_{\Omega_j}({\eta}(U^{n+1}))-\mathcal{Q}_{\Omega_j}({\eta}(U^n))+\mathcal{Q}_t\left(\int_{\partial\Omega_j}\Psi(\hat{U}^-(t,x),\hat{U}^+(t,x))\cdot\hat{n}\dx\right)\right\}. \label{sader}
\end{equation}
\end{defn}

\begin{rem}
Again for a more efficient implementation, $S_j^n$ can be rewritten as
\begin{equation*}
    S_j^n = \dfrac{1}{\Delta t}\left(\dfrac{1}{|\Omega_j|}\mathcal{Q}_{\Omega_j}(\eta(U^{n+1}))-\overline{\eta}_j^{n+1}\right)
\end{equation*}
where $\overline{\eta}_j^{n+1}$ is the entropy evolved as a variable of the system and reinitialized at each time step as $\overline{\eta}_j^n=\eta(\overline{u}_j^n)$.
\end{rem}

Using the definition $\eqref{sader}$, the following proposition can be proved.

\begin{thm}\label{thm:ADERent}
Consider a $p^{th}$-order convergent scheme of the form $\eqref{fv5}$
where $\mathcal{Q}$ is a quadrature rule of order $q$ and define the entropy production as in $\eqref{sader}$ with a consistent entropy flux. Suppose $\Delta t=\mathcal{O}(|\Omega_j|/|\partial\Omega_j|)$. Then
\begin{itemize}
    \item if the solution is regular, the numerical entropy production in a fixed time step decays to zero as $\mathcal{O}(\Delta t^p)$ for $\Delta t\rightarrow 0$;
    \item in those cells in which the solution is not regular, $S_j^n\leq\dfrac{C}{\Delta t}$, for a constant $C$ which does not depend on $\Delta t$.
\end{itemize}
\end{thm}

\begin{proof}
Let $u(t,x)$ be the exact solution and $U(t,x)$ the numerical solution of the conservation law with $t\in[t^n,t^{n+1}]$ and $x=\left(x^{(1)},...,x^{({\revDue \nu})}\right)\in{\revDue \mathbb{R}^{\nu}}$ and suppose that the cell averages of the exact and the numerical solution coincide at the beginning of the timestep:
\begin{equation*}
    \overline{U}_j(t^n)=\overline{u}_j(t^n).
\end{equation*}
Let $\Delta x$ be a scalar quantity related to the diameter of the cell, for example the quantity $|\Omega_j|/|\partial\Omega_j|$.

Firstly, estimate the time contribution in $S_j^n$, namely
$\mathcal{Q}_{\Omega_j}(\eta(U^{n+1}))-\mathcal{Q}_{\Omega_j}(\eta(U^n))$.
We use a reconstruction $\mathcal{R}$ of accuracy $r$ to compute the values of $U^n$ on the quadrature nodes. For smooth solutions it holds
\begin{equation*}
    U^n(x_{\cdot,j}) =  \mathcal{R}(\overline{U}^n;x_{\cdot,j}) = \mathcal{R}(\overline{u}^n,x_{\cdot,j}) = u^n(x_{\cdot,j})+\mathcal{O}(\Delta x^{r+1}). \label{fv6}
\end{equation*}
Applying the Lagrange's mean value theorem
\begin{equation*}
\begin{aligned}
    \eta(U^n(x_{\cdot,j}))-\eta(u^n(x_{\cdot,j}))&=\nabla_u\eta^n \cdot (U^n(x_{\cdot,j})-u^n(x_{\cdot,j}))=\sum_{k=1}^m(U^n_k-u^n_k)\partial_{u_k}\eta^n \\
    &=\sum_{k=1}^m\mathcal{O}(\Delta x^{r+1})\partial_{u_k}\eta^n=D\eta^n\mathcal{O}(\Delta x^{r+1})
\end{aligned}
\end{equation*} 
where $D\eta^n$ is the sum of the first derivatives of $\eta$  with respect to $u$ evaluated at $t^n$.
Using a quadrature rule of order $q$ with nodes $x_{k,j}\in{\revDue \mathbb{R}^{\nu}}$ and weights $w_k$ for $k=1,...,K$, the second term becomes
\begin{equation*}
    \mathcal{Q}_{\Omega_j}(\eta(U^n)=
    |\Omega_j|\sum_{k=1}^Kw_k\eta(U^n(x_{k,j})=
    \mathcal{Q}_{\Omega_j}(\eta(u^n))+\mathcal{Q}_{\Omega_j}\left(D\eta^n\mathcal{O}(\Delta x^{r+1})\right)
\end{equation*}
The numerical solution at time $t^{n+1}$ is computed through the ADER predictor $\eqref{pred}$.  
Using a set of basis functions of degree $M=p-1$ we get
\begin{equation*}
    \hat{U}(t,x)=\tilde{u}(t,x)+\mathcal{O}(\Delta t^{p+1}) 
\end{equation*}
where $\tilde{u}(t,x)$ is the exact solution computed with initial data $\mathcal{R}(\overline{u}^n,x)$ obtained from the reconstruction and if $u(t,x)$ is smooth 
\begin{equation*}
   U^{n+1}(x_{\cdot,j}) =u^{n+1}(x_{\cdot,j})+\mathcal{O}(\Delta x^{r+1})+\mathcal{O}(\Delta t^{p+1})  
\end{equation*}
and we get
\begin{equation*}
    \mathcal{Q}_{\Omega_j}(\eta(U^{n+1}))=
    \mathcal{Q}_{\Omega_j}(\eta(u^{n+1}))
    +\mathcal{Q}_{\Omega_j}\left(D\eta^{n+1}\big(\mathcal{O}(\Delta x^{r+1})+\mathcal{O}(\Delta t^{p+1})\big)\right).
\end{equation*}
Subtracting the two terms, the time contribution becomes
\begin{equation*}
\begin{aligned}
    \mathcal{Q}_{\Omega_j}(\eta(U^{n+1}))-\mathcal{Q}_{\Omega_j}(\eta(U^n))&=
    \mathcal{Q}_{\Omega_j}(\eta(u^{n+1})-\eta(u^n))
    \\
    &\phantom{=\,}+\mathcal{Q}_{\Omega_j}\left((D\eta^{n+1}-D\eta^n)\mathcal{O}(\Delta x^{r+1})
     +\mathcal{O}(\Delta t^{p+1})\right)\\
    &=\mathcal{Q}_{\Omega_j}\left(\int_{t^n}^{t^{n+1}}\partial_t\eta(u)\dt\right)+|\Omega_j|\mathcal{O}(\Delta t\Delta x^{r+1}+\Delta t^{p+1}),
\end{aligned}
\end{equation*}
where we have used the fact that, since the exact solution is smooth,
\begin{equation*}
    D\eta^{n+1}-D\eta^n=\mathcal{O}(\Delta t).
\end{equation*}
Finally, using the approximation of the quadrature rule $\mathcal{Q}_{\Omega_j}$, we get

\begin{equation*}
\begin{aligned}
    \mathcal{Q}_{\Omega_j}&(\eta(U^{n+1}))-\mathcal{Q}_{\Omega_j}(\eta(U^n))=
    \int_{\Omega_j}\int_{t^n}^{t^{n+1}}\partial_t\eta(u)\dt dx\\
    &+C_q|\Omega_j|\Delta x^qD^q\left(\int_{t^n}^{t^{n+1}}\partial_t\eta(u)\dt\right)+|\Omega_j|\mathcal{O}(\Delta x^{q+1}+\Delta t\Delta x^{r+1}+\Delta t^{p+1})\\
    &=\int_{\Omega_j}\int_{t^n}^{t^{n+1}}\partial_t\eta(u)\dt dx+|\Omega_j|\Delta t\mathcal{O}(\Delta x^q+\Delta x^{r+1}+\Delta t^p),
\end{aligned}
\end{equation*}
where $D^q$ denotes the spatial derivatives of order $q$.

Next, we estimate the space contribution $\mathcal{Q}_t\left(\mathcal{Q}_{\partial\Omega_j}\left(\Psi(\hat{U}^-,\hat{U}^+)\cdot\hat{n}\right)\right)$.
The quadrature over $\partial\Omega_j$ can be written as the sum of the integrals over the faces of $\Omega_j$, and using the Lipschitz-continuity and the consistency of the numerical entropy flux we obtain 
\begin{equation*}
\begin{aligned}
    \mathcal{Q}_{\partial\Omega_j}\left(\Psi(\hat{U}^-,\hat{U}^+)\cdot\hat{n}\right)&=
    \sum_{e\in\partial\Omega_j}\mathcal{Q}_{e}\left(\Psi(\hat{U}^-,\hat{U}^+)\cdot\hat{n}_e\right) \\
    =&\sum_{e\in\partial\Omega_j}\mathcal{Q}_{e}\left(\psi(u)\cdot\hat{n}_e\right)+|\partial\Omega_j|\mathcal{O}(\Delta t^{p+1}).
\end{aligned}
\end{equation*}
Now, using approximation of the quadrature rule, we get for each face
\begin{equation*}
   \begin{aligned}
    \mathcal{Q}_{e}\left(\psi(u)\cdot\hat{n}_e\right)&=
    \int_{e}\psi(u)\cdot\hat{n}_e
    +C_q|e|\Delta x^q D^{q}\psi(\xi_e)\hat{n}_e 
    +\mathcal{O}(|e|\Delta x^{q+1})\\
    &=\int_{e}\psi(u)\cdot\hat{n}_e
    +C_q\Delta x^{(q+{\revDue \nu}-1)} D^{q}\psi(\xi_e)\hat{n}_e
    +\mathcal{O}(|e|\Delta x^{q+1})\\
    &=\int_{e}\psi(u)\cdot\hat{n}_e
    +C_q\Delta x^{(q+{\revDue \nu}-1)} D^{q}\psi(b_j)\hat{n}_e
    +\mathcal{O}(|e|\Delta x^{q+1}).
\end{aligned} 
\end{equation*}
In the last line, we have used the fact that the evaluation point $\xi_e$ on the face is $\Delta x$ apart from the barycenter $b_j$ of the cell. Summing up the contributions from all faces of $\Omega_j$ and using that $\sum_{e\in\partial\Omega_j}\hat{n}_e=0$, we get
\begin{equation*}
\mathcal{Q}_{\partial\Omega_j}\left(\Psi(\hat{U}^-,\hat{U}^+)\cdot\hat{n}\right)
=\int_{\Omega_j}\nabla_x\cdot\psi(u)\dx+|\partial\Omega_j|\mathcal{O}(\Delta x^{q+1}+\Delta t^{p+1}).
\end{equation*}
Applying the quadrature rule in time, we obtain
\begin{equation*}
   \begin{aligned}
    \mathcal{Q}_t\left(\mathcal{Q}_{\partial\Omega_j}\left(\Psi(\hat{U}^-,\hat{U}^+)\cdot\hat{n}\right)\right)
    &=\int_{t^n}^{t^{n+1}}\!\!\!\int_{\Omega_j}\nabla_x\cdot\psi(u)\dx dt+\mathcal{O}(\Delta t^{q+2}) \\
    & \phantom{=\,} +\Delta t|\partial\Omega_j|\mathcal{O}(\Delta x^{q+1}+\Delta t^{p+1}).
    \end{aligned}
\end{equation*}

Finally, summing up the time and the spatial part we get
\begin{equation*}
\begin{aligned}
    S_j^n&=\frac{1}{\Delta t|\Omega_j|}\int_{t^n}^{t^{n+1}}\!\!\!\int_{\Omega_j}\partial_t\eta(u)\dx dt+\mathcal{O}(\Delta x^q+\Delta x^{r+1}+\Delta t^p)\\
    &+\dfrac{1}{\Delta t|\Omega_j|}\int_{t^n}^{t^{n+1}}\!\!\!\int_{\Omega_j}\nabla_x\cdot\psi(u)\dx dt+\dfrac{\Delta t|\partial\Omega_j|}{|\Omega_j|}\mathcal{O}(\Delta x^{q}+\Delta t^{p}).
\end{aligned}
\end{equation*}
Therefore, choosing $r+1\geq p$ and $q\geq p$,
$S_j^n=\mathcal{O}(\Delta t^p)$ for smooth solutions.

For the second statement, it is enough to note that every term in $\eqref{sader}$ is bounded. Therefore $|S_j^n|\leq C/\Delta t$, where $C$ does not depend on $\Delta t$.
\end{proof}

The proof of the theorem considers only the extremal cases of a smooth data or a generic input, which is intended to represent the worst case scenario of a shock. We now compute in more detail the values assumed by the numerical entropy production in few cases, to illustrate better the results of the theorem and to explore intermediate flow regularity. 

We perform the computations in the case of a second order $P_0P_1$ scheme.
In particular we look for a predictor in the Galerkin space $\mathbb{P}_1(\tau)\otimes\mathbb{P}_1(\xi)$, and then compute the fluxes at cell boundaries by integrating in time using the trapezoidal rule. We choose {\revDue $\Delta t$} so that the relevant wave in each case crosses half of a cell.
The scheme is also simplified, assuming an exact reconstruction, i.e. that the boundary value reconstructions coincide with the initial data evaluated at the cell boundary. Unless in some cases (specified below) where the predictor's solution can be computed analytically, the predictor is computed by applying two iterations of \eqref{eq:iter}.

\paragraph{Example: smooth data}
Let us consider a smooth initial data, whose Taylor expansion around a cell center $x_j$ is $u(x)=u_j+u_x(x-x_j)+\tfrac12 u_{xx}(x-x_j)^2+\mathcal{O}(x-x_j)^3$. Applying one step of the second order numerical scheme, one obtains
\[
S_j = \frac{u_x u_{xx}}{2} (\Delta x)^2 + \mathcal{O}(\Delta x)^3
\]
in the case of the linear transport equation $\partial_tu+\partial_xu=0$ and
\[
S_j = \left(\frac{13}{24}  u_x^2 + \frac{1}{2}u_ju_xu_{xx}\right)(\Delta x)^2 + \mathcal{O}(\Delta x)^3
\]
for Burgers'.
As expected, in both cases we have obtained $S=\mathcal{O}(\Delta x^2)$, corresponding to the behaviour of the second order numerical scheme on a smooth flow. We also observe that, both for linear transport and Burgers',  the leading $\mathcal{O}(\Delta x^2)$ term vanishes at local extrema; we point out that the case of local extrema is a point of difficulty when studying the sign of the numerical entropy production also in simpler cases, see e.g. \cite{PS11:numerical:entropy}.

\paragraph{Example: shock wave}
Next let us consider a discontinuous initial data 
\[
    u_0(x)=
    \begin{cases}
    u_L & \mbox{for } x<-\tfrac{\Delta x}{2}\\
    u_R & \mbox{for } x>-\tfrac{\Delta x}{2}
    \end{cases}
\]
with $u_L>u_R=0$ and define the jump $\delta = u_L-u_R>0$ across the shock. 
Here computations are made easier since one knows that the predictor in each cell will be constant in both space and time and one has only to apply the numerical flux function and the trapezoidal quadrature rule.
For Burgers' one obtains
\[
    S = \frac{1}{\Delta x}\left(-\frac{5}{6}\delta^3+\frac{9}{16}\frac{\Delta t}{\Delta x}\delta^4\right) = \frac{C}{\Delta x}
\]
with $C\sim\delta^3$, when choosing $\frac{\Delta t}{\Delta x}=\frac{1}{u_L}=\frac{1}{\delta}$ so that the shock, which moves at speed $(u_L+u_R)/2$, travels for half a cell in a timestep. This corresponds to the results on the exact entropy production at a shock in the Burger's equation \cite{LeVeque:book}.

\paragraph{Example: kink and corner of rarefaction wave}
Finally, as an illustration of the behaviour of $S$ in the case of an intermediate regularity situation, let us consider an initial data with a kink, i.e. continuous but not differentiable. Consider 
\[
    u_0(x)=
    \begin{cases}
    \tilde{u} & \mbox{for } x<-\tfrac{\Delta x}{2}\\
    \tilde{u}+u_x(x+\tfrac{\Delta x}{2}) & \mbox{for } x\geq-\tfrac{\Delta x}{2}
    \end{cases}
\]
We study the numerical entropy production in the first timestep in the cell $[-\tfrac{\Delta x}{2},\tfrac{\Delta x}{2}]$.
The predictor on the cell $[-\tfrac{3\Delta x}{2},-\tfrac{\Delta x}{2}]$ is constant in space and time.
For the linear transport (with a positive velocity) of this data, 
the predictor for the cells 
$[-\tfrac{\Delta x}{2},\tfrac{\Delta x}{2}]$
and
$[\tfrac{\Delta x}{2},\tfrac{3\Delta x}{2}]$
coincides with exact transport of the linear initial data. 
One gets
\[
    S=\frac{29}{32}u_x^2\text{ }\Delta x+\mathcal{O}(\Delta x^2)
\]
In the case of Bugers', when $\hat{u}$ and $u_x$ are non-negative, it gives rise to a right-moving rarefaction wave. The numerical entropy production is
\[
    S=\frac{29}{32}\tilde{u}u_x^2\text{ }\Delta x + \mathcal{O}(\Delta x^2)
\]
In both cases the numerical entropy production is $S=\mathcal{O}(\Delta x)$, and this corresponds to the scheme having an order of accuracy degraded to 1 on this cell. Of course, in the special case $u_x=0$, no discontinuity is present and one recovers the case of a smooth solution.

\section{Numerical tests on properties of the entropy production}

In this section, we show numerical evidence of the properties of the numerical entropy production by FV-ADER schemes.
We consider a uniform Cartesian mesh. 
We consider in particular second order $P_0P_1$ schemes and third order $P_0P_2$ schemes, which compute the cell averages at time $t^{n+1}$ by using a local $\mathbb{P}_1\otimes\mathbb{P}_1$ and $\mathbb{P}_2\otimes\mathbb{P}_2$ space-time continuous Galerkin predictor. The time integrals of the numerical fluxes for the evolution are computed with the trapezoidal and the Cavalieri-Simpson rule, respectively.

\subsection{One-dimensional case}

We consider Euler equations of gasdynamics. The unknowns are $u = [\rho, \rho v, E]$ where $\rho$ is the density, $v$ the velocity and $E = \tfrac{1}{2}\rho v^2 + p/(\gamma-1)$ is the total energy. We consider $\gamma = 1.4$. The flux function is $f(u) = [\rho v, \rho v^2+p, v(E+p)]$. We take as entropy pair the entropy function $\eta(u) = -\rho \log(\tfrac{p}{(\gamma-1)\rho^{\gamma}}$) and the entropy flux $\psi(u) = v\eta(u)$.

The fully-discrete $P_0P_M$ scheme is implemented with $M = 1$ and $M = 2$, reaching respectively a second and a third order of accuracy. We choose Rusanov flux and a piecewise linear polynomial with minmod slopes \cite{Sweby:84} for the second order reconstruction and a CWENOZ reconstruction \cite{CSV19:cwenoz} of order 3 for the third order scheme. We fix the CFL number at 0.5.

\paragraph{Convergence test}
We take as domain the interval $[0, 1]$ and as initial data the smooth function
\begin{equation*}
	u_0(x)=
	\begin{cases}
		\rho(x)=1+0.5\sin(2\pi x)\\
		v(x)=1\\
		p(x)=1
	\end{cases}
\end{equation*}
with periodic boundary conditions and final time $t_{Max} = 0.1$.

Table~\ref{tab:convergence} shows the rate of convergence of the $P_0P_1$ and $P_0P_2$ schemes compared with the rate of convergence to zero of $S_j^n$ in norm 1, under grid refinement. Notice that refining the grid, $S_j^n$ scales as the order of the corresponding numerical scheme. The results observed are consistent with the theory.
Fig.~\ref{fig1} shows the plot of the density compared to the corresponding entropy production by a $P_0P_2$ scheme with $N=128,256,512$. Notice in particular the negative sign of $S_j^n$.

\begin{table}
\begin{center}
\begin{filecontents*}{errori1.txt}

64,    0.002238358342339,       ,   0.004571705105974,        ,  0.032842672850347,        ,  0.000195161351228, 
128,   0.000599580912875,  1.9004,  0.001093465697182,  2.0638,  0.004050548874457,  3.0194,  0.000024279039418,  3.0069
256,   0.000160470771318,  1.9016,  0.000266715559229,  2.0355,  0.000502768107570,  3.0102,  0.000003034769787,  3.0001
512,   0.000042710694942,  1.9096,  0.000065925265648,  2.0164,  0.000062851081213,  2.9999,  0.000000379366445,  2.9999
1024,  0.000011184520151,  1.9331,  0.000016398774688,  2.0072,  0.000007856515588,  3.0000,  0.000000047421749,  3.0000
2048,  0.000002904762282,  1.9450,  0.000004090718396,  2.0032,  0.000000981174660,  3.0013,  0.000000005927719,  3.0000

%32, 0.00738667166972258, , 0.0197408485783227, , 0.000297902069501119, , 0.00169985178961066, 
%64, 0.00217888647336637, 1.76133342310500, 0.00457170510598008, 2.11037975485765, 3.21132230906712e-05, 3.21359877026166, 0.000195161351234719, 3.12266968071878
%128, 0.000594074087160211, 1.87487626330223, 0.00109346569716906, 2.06382438403376, 4.01849403809836e-06, 2.99844062392111, 2.42790394289865e-05, 3.00688412782973
%256, 0.000160232000680669, 1.89048055730162, 0.000266715559231868, 2.03553406855551, 5.02344286679776e-07, 2.99990656775234, 3.03476978331866e-06, 3.00005236087496
%512, 4.26615829273408e-05, 1.90915290193609,	6.59252656637019e-05, 2.01639860406860, 6.27975778014385e-08, 2.99989565028293, 3.79366466233155e-07, 2.99992301450766

%32, 0.0438377276253461, , 0.200393863730279, , 0.00738667166972258, , 0.0197408485783227, , 
%64, 0.0163065487137448, 1.42672154913253, 0.0432604338469786, 2.21171829350188, 0.00217888647336637, 1.76133342310500, 0.00457170510598008, 2.11037975485765,
%128, 0.00476398898416969, 1.77520948513435, 0.0100061210409488, 2.11216532437961, 0.000594074087160211, 1.87487626330223, 0.00109346569716906, 2.06382438403376,
%256, 0.00129775906958263, 1.87614750699016, 0.00238181554473910, 2.07074921356913, 0.000160232000680669, 1.89048055730162, 0.000266715559231868, 2.03553406855551,
%512, 0.000350734140803985, 1.88757279437047, 0.000580509846915139, 2.03666924653479, 4.26615829273408e-05, 1.90915290193609,	6.59252656637019e-05, 2.01639860406860, 
%1024, 9.33509014636382e-05, 1.90964201143583, 0.000143344687271110, 2.01783210353584, 1.11796167623495e-05, 1.93206676344043, 1.71998821429956e-05, 1.93843280050307, 
\end{filecontents*}
\pgfplotstabletypeset[
		col sep=comma,
		sci zerofill,
		empty cells with={--},
		every head row/.style={before row=\toprule,after row=\midrule},
            every first column/.style={column type/.add={}{|}},
            every fifth column/.style={column type/.add={|}{}},
		every last row/.style={after row=\bottomrule},
		create on use/rate/.style={create col/dyadic refinement rate={1}},
		columns/0/.style={column name={N}},
            columns/1/.style={column name={$P_0P_1$}},
            columns/2/.style={column name={Rate},fixed zerofill},
            columns/3/.style={column name={S}},
            columns/4/.style={column name={Rate},fixed zerofill},
            columns/5/.style={column name={$P_0P_2$}},
            columns/6/.style={column name={Rate},fixed zerofill},
            columns/7/.style={column name={S}},
            columns/8/.style={column name={Rate},fixed zerofill},
		columns={0,1,2,3,4,5,6,7,8},
		]
		{errori1.txt}
\end{center}
\caption{Rate of convergence of the $P_0P_1$ and $P_0P_2$ schemes compared with the decay to zero of $S^n$ at each time-step.}
\label{tab:convergence}
\end{table}

\begin{figure}[h]
	\centering
        \includegraphics[width=0.9\textwidth]{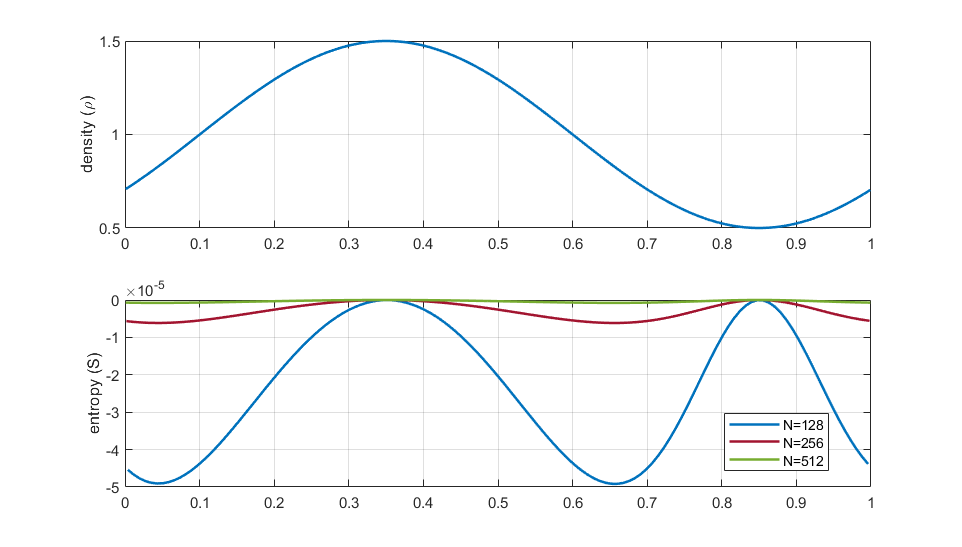}
\caption{Smooth solution using $P_0P_2$ and corresponding numerical entropy production with a grid of $N=128,256$ and $512$ cells.}
\label{fig1}
\end{figure}

\vspace{0.5cm}
Next, we analyze the behaviour of the numerical entropy production by the $P_0P_1$ and $P_0P_2$ schemes applied to initial data which present a discontinuity in $x_0$ and which give rise to contact discontinuities, shock or rarefaction waves.

\paragraph{Rarefaction waves}
We take as domain the interval $[-2,2]$ and as initial data
\begin{equation*}
	u_0(x)=
	\begin{cases}
		u_L \text{ if }x\leq0\\
		u_R \text{ if }x>0
	\end{cases}
\end{equation*}
where the conservative variables $u_L, u_R$ represent the states $(\rho_L,v_L,p_L)=(1,-0.15,1)$, $(\rho_R,v_R,p_R)=(0.5,0.15,1)$
and final time $t_{Max}  = 0.5$. This initial data gives rise to two rarefaction waves that move in opposite directions. 

\begin{figure}
	\centering
	\includegraphics[width=0.48\textwidth]{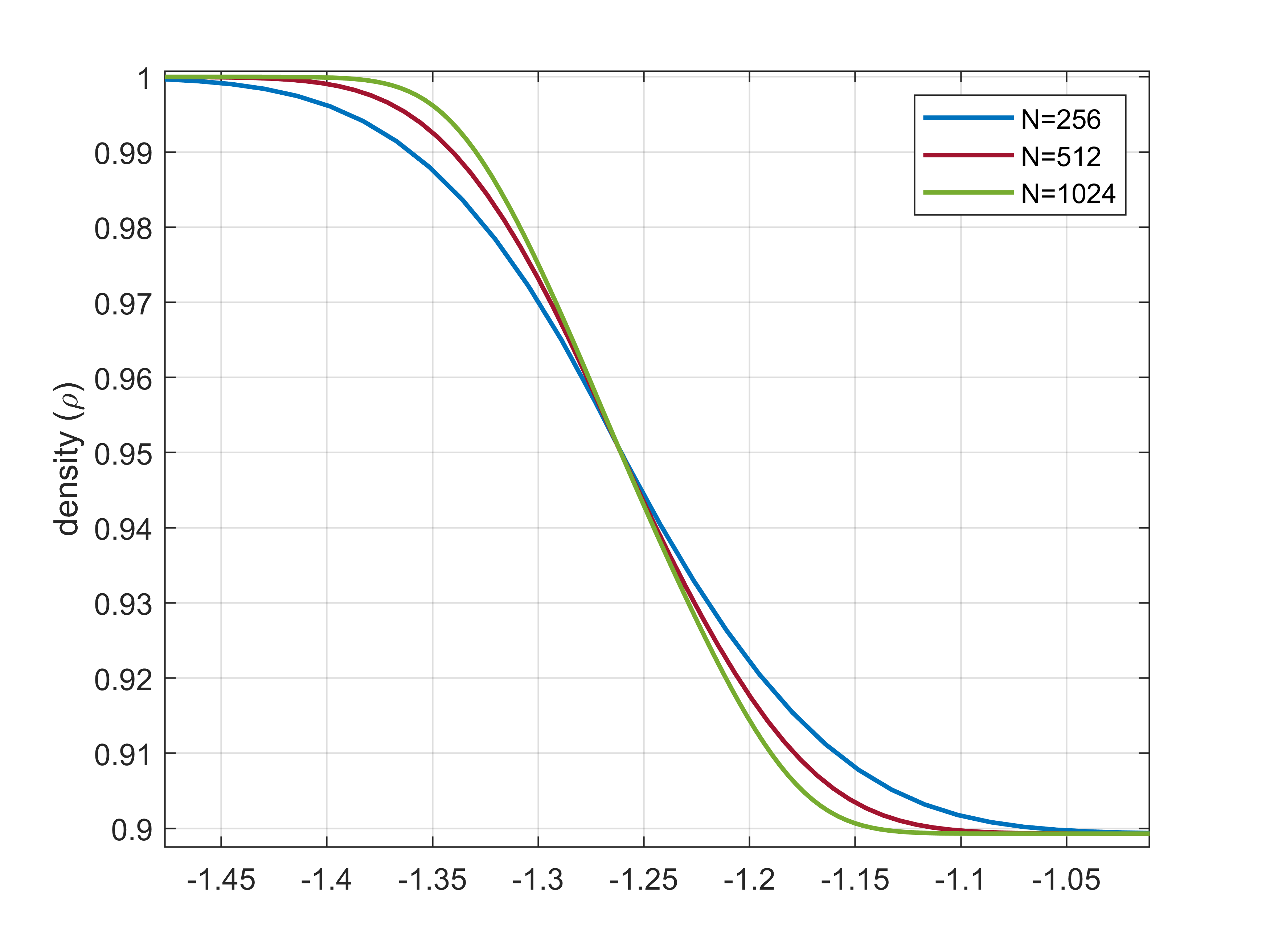}
    \hfill
    \includegraphics[width=0.46\textwidth]{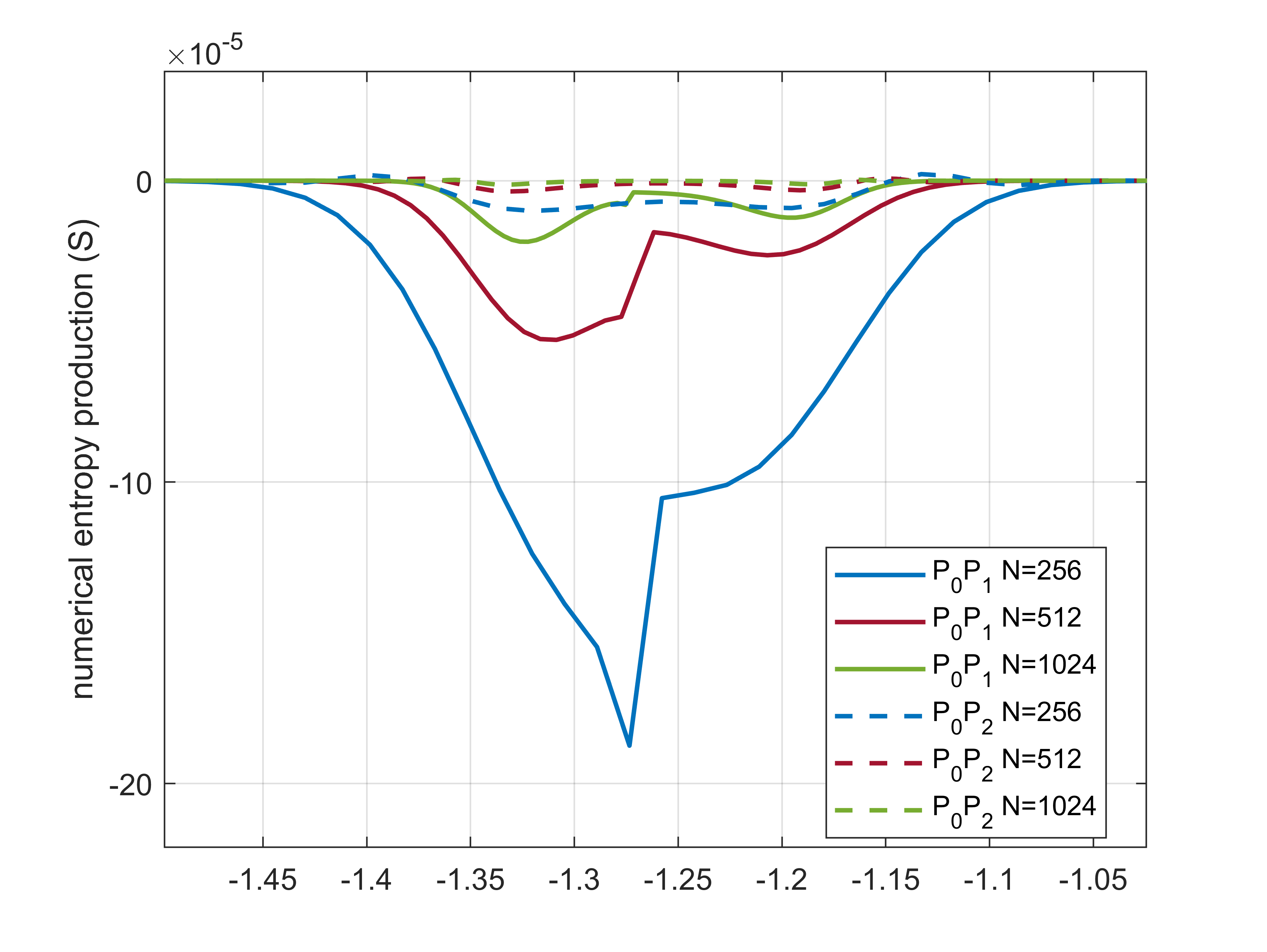}
\caption{Left rarefaction wave computed by $P_0P_1$ (left) and entropy production by $P_0P_1$ and $P_0P_2$ (right) with a grid of $N=256,512$ and 1024 cells.}
\label{fig2}
\end{figure}

In Fig.~\ref{fig2} the numerical solution of the density (left panel) and the corresponding numerical entropy production (right panel) for the left rarefaction wave are compared. Under grid refinement, the scheme improves the resolution of the density and $S_j^n$ has a small essentially negative value that converges to zero. We observe that the absolute values of the numerical entropy production are smaller for the third order scheme and that these latter also show some very small positive overshoots.

\paragraph{Contact discontinuities}
Then we consider as initial data
the Riemann problem with left and right states determined by $(\rho_L,v_L,p_L)=(2,0.1,1)$ and $(\rho_L,v_L,p_L)=(1,0.1,1)$,
as domain the interval $[-5,5]$ and as final time $t_{Max}  = 10$. This initial data gives rise to a contact discontinuity that moves slowly starting from $x_0 = 0$. 

\begin{figure}
	\centering
        \includegraphics[width=0.48\textwidth]{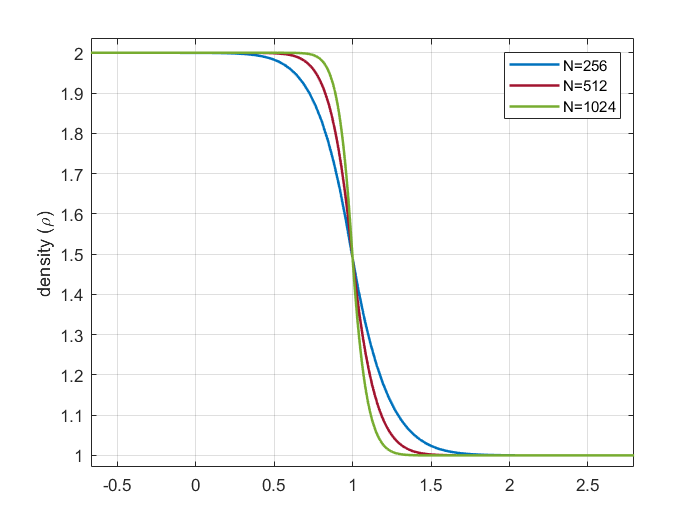} 
        \hfill
        \includegraphics[width=0.46\textwidth]{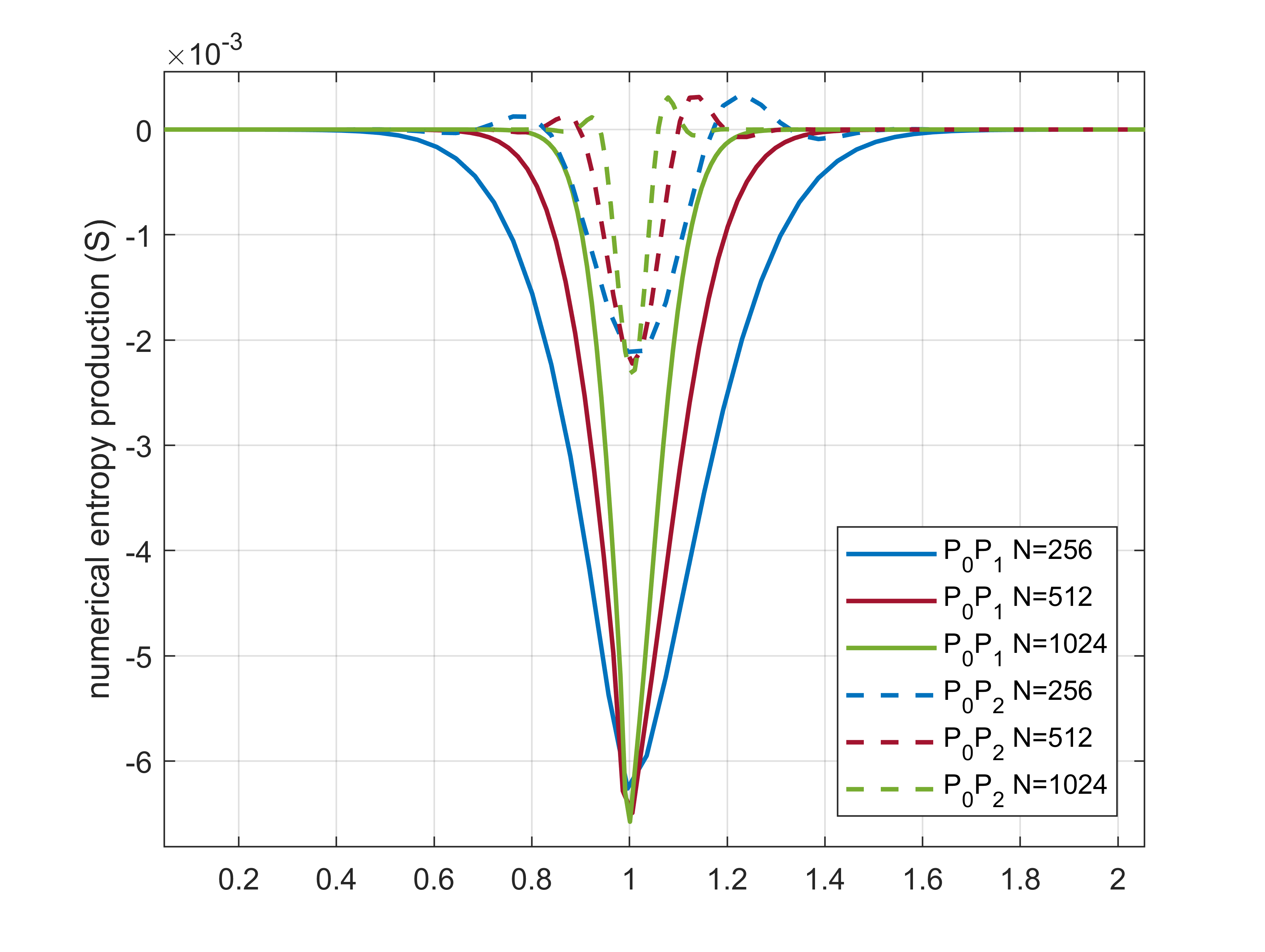}
\caption{Contact discontinuity wave computed by $P_0P_1$ (left) and entropy production of $P_0P_1$ and $P_0P_2$ (right) with a grid of $N=256,512$ and 1024 cells.}
\label{fig3}
\end{figure}

In Fig.~\ref{fig3} we show the numerical solution of the density (left panel) and the corresponding numerical entropy production (right panel). For each scheme, $S_j^n$ remains approximately constant across the discontinuity (at $x=1$) under grid refinement. Notice that the order of magnitude is increased with respect to the rarefaction wave. We observe that also in this case the absolute value of the numerical entropy production is smaller for the third order scheme and that there are some positive overshoots near the discontinuity.

\paragraph{Shock waves}
Next we compute the Riemann problem
with left and right states determined by 
$(\rho_L,v_L,p_L)=(1.5,4,10)$ and $(\rho_R,v_R,p_R)=(0.5,-4,10)$, 
with domain $[-3,7]$ and final time $t_{Max}  = 1$. This Riemann problem produces two shock waves that propagate in opposite directions. Consider the right shock wave and analyze the behaviour of $S_j^n$. 
 
\begin{figure}
	\centering
        \includegraphics[width=0.48\textwidth]{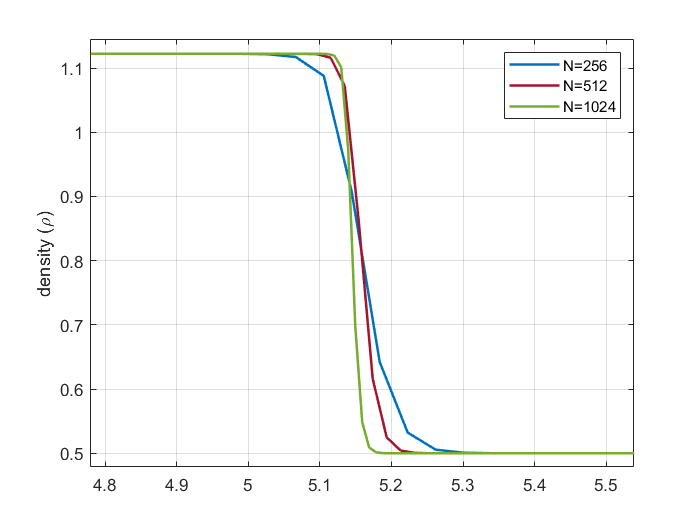}
        \hfill
        \includegraphics[width=0.46\textwidth]{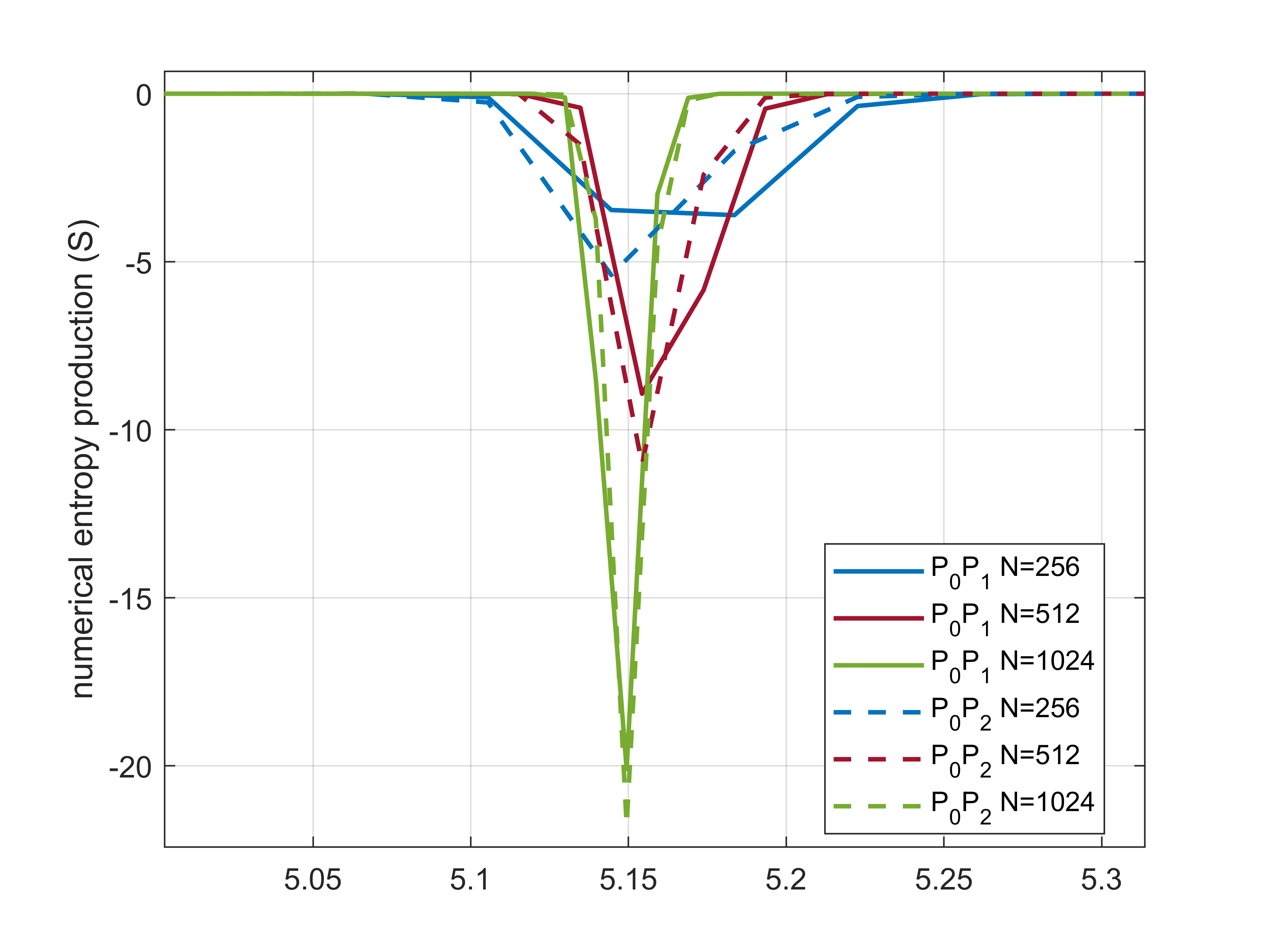}
\caption{Right shock wave computed by $P_0P_1$ (left) and entropy production of $P_0P_1$ and $P_0P_2$ (right) with a grid of $N=256,512$ and 1024 cells.}
\label{fig4}
\end{figure}

In Fig.~\ref{fig4} we show the numerical solution of the density (left panel) and the corresponding numerical entropy production (right panel). $S_j^n$ increases as $1/\Delta x$ under grid refinement and, since $\Delta t=\mathcal{O}(\Delta x)$, as $1/\Delta t$. Notice that the magnitude of the numerical entropy is significantly higher than the previous cases. In the case of shocks, there is no significant difference in magnitude for the absolute value of the numerical entropy production between second and third order schemes. This is consistent with the fact that in this case the entropy production is non-zero also on the exact solution, whereas in all other cases the entropy production is zero on the exact solution and the numerical entropy production is thus just a truncation error.

\paragraph{Sod problem}
Finally, we compare $S_j^n$ on the Sod tube problem in order to get a general overview on the different kind of discontinuities. We take as domain the interval $[0,1]$ and as initial data
\begin{equation*}
	u_0(x)=
	\begin{cases}
		u_L \text{ if }x\leq0.5\\
		u_R \text{ if }x>0.5
	\end{cases}
\end{equation*}
where $(\rho_L,v_L,p_L)=(1,0,1)$, $(\rho_L,v_L,p_L)=(0.125,0,0.1)$
and final time $t_{Max} = 0.2$. 

\begin{figure}
	\centering
	\includegraphics[width=0.9\textwidth]{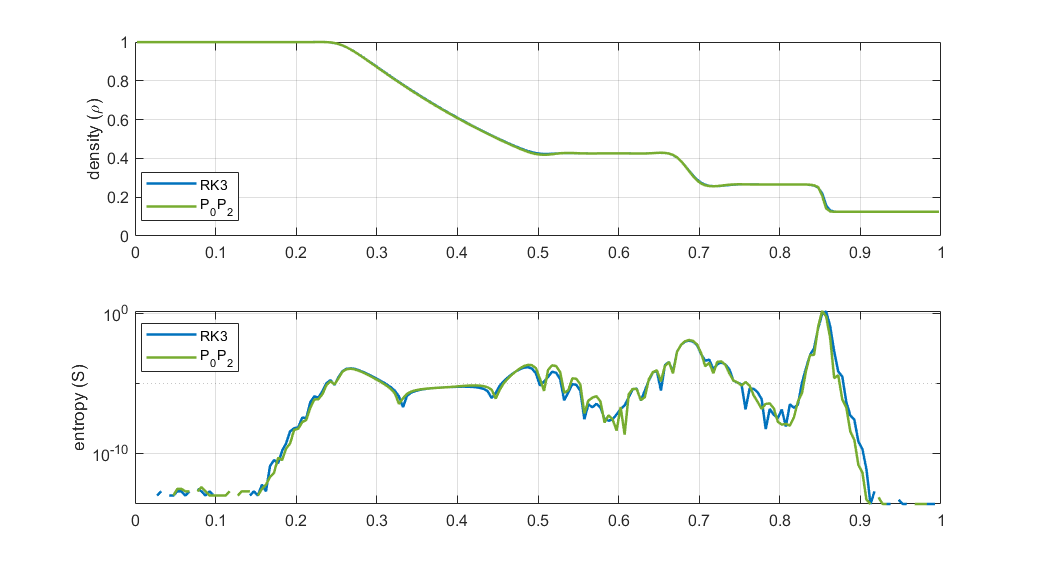}
\caption{Sod tube problem and corresponding entropy production in logarithmic scale by 3-Runge-Kutta FV and $P_0P_2$ with a grid of 100 cells.}
\label{fig5}
\end{figure}

Fig.~\ref{fig5} we plot the numerical solution of the density ( top panel) and the corresponding numerical entropy production (bottom panel) using a Runge Kutta method of order 3 and a $P_0P_2$ scheme with $N=100$. We observe that the  numerical entropy is almost zero in smooth regions, with a slightly bigger production in correspondence of the corners of the rarefaction wave (see also the example after Theorem~\ref{thm:ADERent}) and with a more pronounced signal at the contact discontinuity. The highest numerical entropy production is localized on the shock wave.
The plot also shows that the numerical entropy production of both the Runge-Kutta-FV and the FV-ADER scheme is comparable in all regions of the solution.

\subsection{Two-dimensional case}
In this section we present some numerical tests in the two-dimensional case. The numerical schemes are run using PETSc libraries for grid management and parallel computing \cite{petsc-user-ref}. 
We use a $P_0P_2$ scheme, with CWENO reconstruction \cite{CPSV:cweno} and the Rusanov approximate Riemann solver.  We fix the CFL number at 0.45.

We consider 2D Euler equations, with unknown $\vec{u}=[\rho,\rho u,\rho v, E]$ and fluxes $f(\vec{u}) =[\rho u,\rho u^2+p,\rho u v,u(E+p)] $ and $g(\vec{u})=[\rho v, \rho u v, \rho v^2+p,v(E+p)]$. The entropy pair is taken as $\eta(\vec{u})=-\rho\log\left(\tfrac{p}{(\gamma-1)\rho^\gamma}\right)$ and $\psi(\vec{u})=[u\eta(\vec{u}),v\eta(\vec{u})]$.

\paragraph{Convergence test}
Consider the isentropic vortex solution proposed in \cite{shu:ICASE}
\begin{equation*}
	\begin{cases}
		\rho(x,y) = \rho_{\infty}\left(\dfrac{T}{T_{\infty}}\right)^{\tfrac{1}{\gamma-1}}\\
	    u(x,y) = u_{\infty} - \dfrac{\beta y}{2\pi}\exp\left({\dfrac{1-r^2}{2}}\right)\\
        v(x,y)=v_{\infty}+\dfrac{\beta x}{2\pi}\exp\left({\dfrac{1-r^2}{2}}\right)\\
		p(x,y)=\rho T
	\end{cases}
\end{equation*}
where $r=\sqrt{x^2+y^2}$, $u_{\infty}=v_{\infty}=p_{\infty}=T_{\infty}=1$,  $\rho_{\infty}=p_{\infty}/T_{\infty}$,  $\beta=5$ and
\begin{equation*}
    T = T_{\infty}-\dfrac{(\gamma-1)\beta^2}{8\gamma\pi^2}\exp(1-r^2)
\end{equation*}
in the domain $[-5,5]^2$. We fix $t_{Max}=10$ in order to make it complete a period. We impose periodic boundary conditions.

Table~\ref{tab:conv:2d} shows the rate of convergence of the  $P_0P_2$ scheme compared to the rate of convergence to zero of the numerical entropy $|S_{j}^n|$ under grid refinement, both computed in 1-norm. We compare the results with the data obtained from a semidiscrete Runge-Kutta method of third order of accuracy. As expected, $S_j^n$ has again the property of reflecting the local truncation error of the scheme.
Fig.~\ref{fig6} shows the density computed with a grid of $512\times512$ cells (left panel), the 1-norm of the error  (central panel) and the corresponding numerical entropy production (right panel). Notice that the numerical entropy is greater in absolute value in the regions in which the density error is higher and it is smaller where the error is lower.

\begin{table}
\begin{center}
\begin{filecontents*}{errori2d.txt}
32, 7.5662e-01, , 2.2605e-02, , 7.1637e-01, , 2.2520e-02, 
64, 1.7341e-01, 2.12, 8.7617e-03, 1.38, 1.6531e-01, 2.62, 8.3731e-03, 1.43
128, 3.1391e-02, 2.47, 1.4108e-03, 2.63, 2.9664e-02, 1.97, 1.3439e-03, 2.64
256, 4.5425e-03, 2.79, 1.8410e-04, 2.93, 4.2899e-03, 2.79, 1.7538e-04, 2.94
512, 6.0394e-04, 2.91, 2.3100e-05, 2.99, 5.7106e-04, 2.91, 2.2462e-05, 2.96
% 1024, 7.6816e-05, 2.97, 2.8882e-06, 3, , , , 
\end{filecontents*}
\pgfplotstabletypeset[
		col sep=comma,
		empty cells with={--},
		every head row/.style={before row=\toprule,after row=\midrule},
		every last row/.style={after row=\bottomrule},
		create on use/rate/.style={create col/dyadic refinement rate={1}},
		columns/0/.style={column name={N}},
		columns/1/.style={column name={RK3}},
            columns/2/.style={column name={Rate}},
            columns/3/.style={column name={S}},
            columns/4/.style={column name={Rate}},
            columns/5/.style={column name={$P_0P_2$}},
            columns/6/.style={column name={Rate}},
            columns/7/.style={column name={S}},
            columns/8/.style={column name={Rate}},
		columns={0,1,2,3,4,5,6,7,8},
		]
		{errori2d.txt}
\end{center}
\caption{Rate of convergence of RK3 and $P_0P_2$ compared with the decay to zero of $S^n$ at each time-step.}
\label{tab:conv:2d}
\end{table}

\begin{figure}[h]
	\centering
	\includegraphics[width=0.31\textwidth]{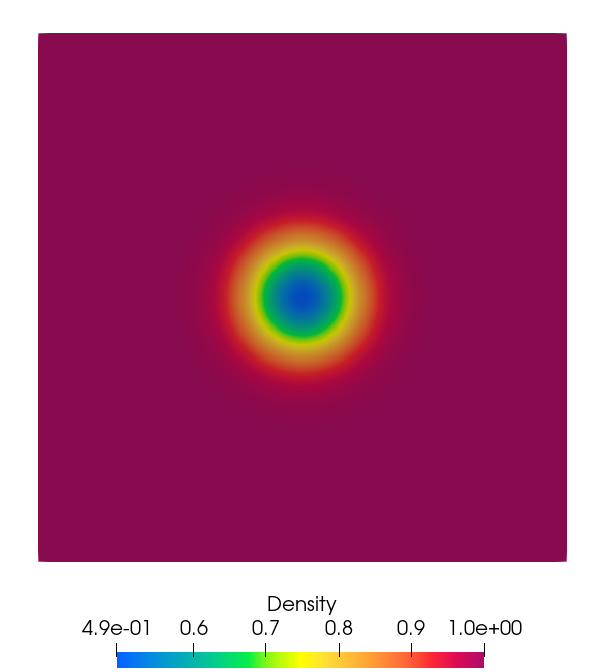}
    \hfill
    \includegraphics[width=0.31\textwidth]{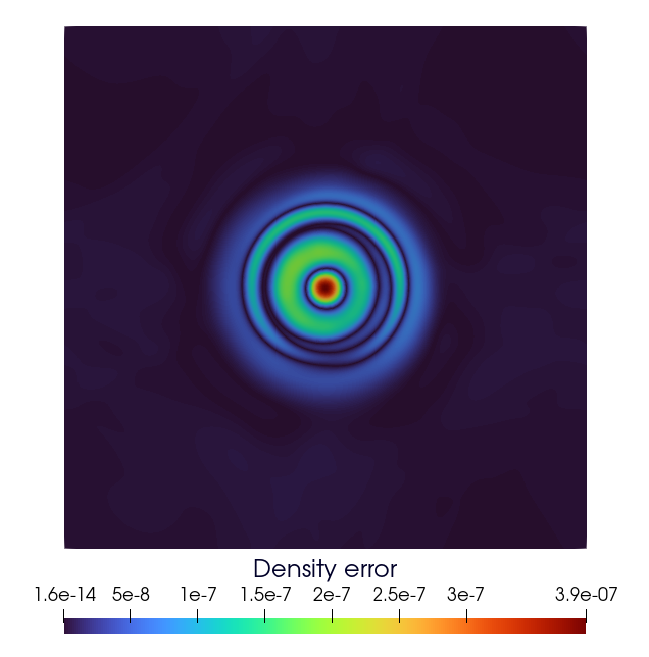}
    \hfill
    \includegraphics[width=0.31\textwidth]{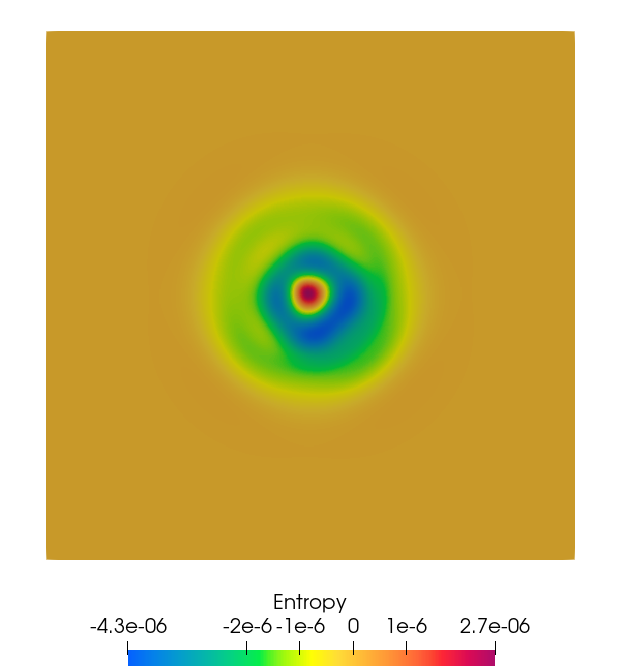}
\caption{Isentropic vortex: density, density error and numerical entropy.}
\label{fig6}
\end{figure}

\paragraph{Radial Sod problem}
Next we test the radial Sod problem in 2D, with initial conditions 
\[
\left\{
\begin{array}{lll}
 u = 0, & v=0 & \\
\rho_H = 1, & p_H=1 & \mbox{ if } x^2+y^2<0.25 \\
\rho_L=0.125, & p_L=0.1 & \mbox{ if } x^2+y^2\geq0.25
\end{array}
\right.
\]
For reasons of symmetry, we compute the solution only in the first quadrant of the domain and we consider as computational domain $[0,1]^2$ with $t_{Max}=1$, applying symmetry boundary conditions along the coordinate axis and wall boundary conditions on the outer sides.  
Firstly, we compare the solutions at $t=0.2$ using a grid of $400\times400$ cells (Fig.\ref{sod}).
We observe that, as in the one dimensional case, we can recognize the type of discontinuity just looking at the numerical entropy production. Indeed, the maximum value of $S$ is located at the shock, whereas on the contact discontinuity it is much smaller and it is almost zero on the rarefaction wave.
Letting the simulation run for a longer time, the shock wave interacts with the wall and then with the contact discontinuity. We show the solution at $t=0.6$ in Fig.\ref{sod}. Notice that also in this more complicated setting, the numerical entropy production allows to easily recognize the type of discontinuity we are dealing with.

Next we run the same test doubling the number of grid points from $32$ to $512$. In Table~\ref{tab:entShockSod}, the maximum of $|S_{ij}^n|$ is shown, which is the value of the numerical entropy production on the shock wave. Observe that it is growing as $\mathcal{O}(1/\Delta x)$, as we expected.

\begin{table}
\begin{center}
\begin{filecontents*}{sodshock.txt}
32,    2.723e-01, 
64,    4.7075e-01,  -0.79
128,   7.7228e-01,  -0.71
256,   1.721,       -1.16
512,   3.72,        -1.11
%1024,  6.42,        -0.78
\end{filecontents*}
\pgfplotstabletypeset[
		col sep=comma,
		empty cells with={--},
		every head row/.style={before row=\toprule,after row=\midrule},
		every last row/.style={after row=\bottomrule},
		create on use/rate/.style={create col/dyadic refinement rate={1}},
		columns/0/.style={column name={N}},
		columns/1/.style={column name={$P_0P_2$}},
            columns/2/.style={column name={Rate}},
		columns={0,1,2},
		]
		{sodshock.txt}
\end{center}
\caption{Infinity norm of $S^n$ by $P_0P_2$.}
\label{tab:entShockSod}
\end{table}

\begin{figure}[h]
\centering
\begin{tabular}{cc}
\includegraphics[width=0.47\textwidth]{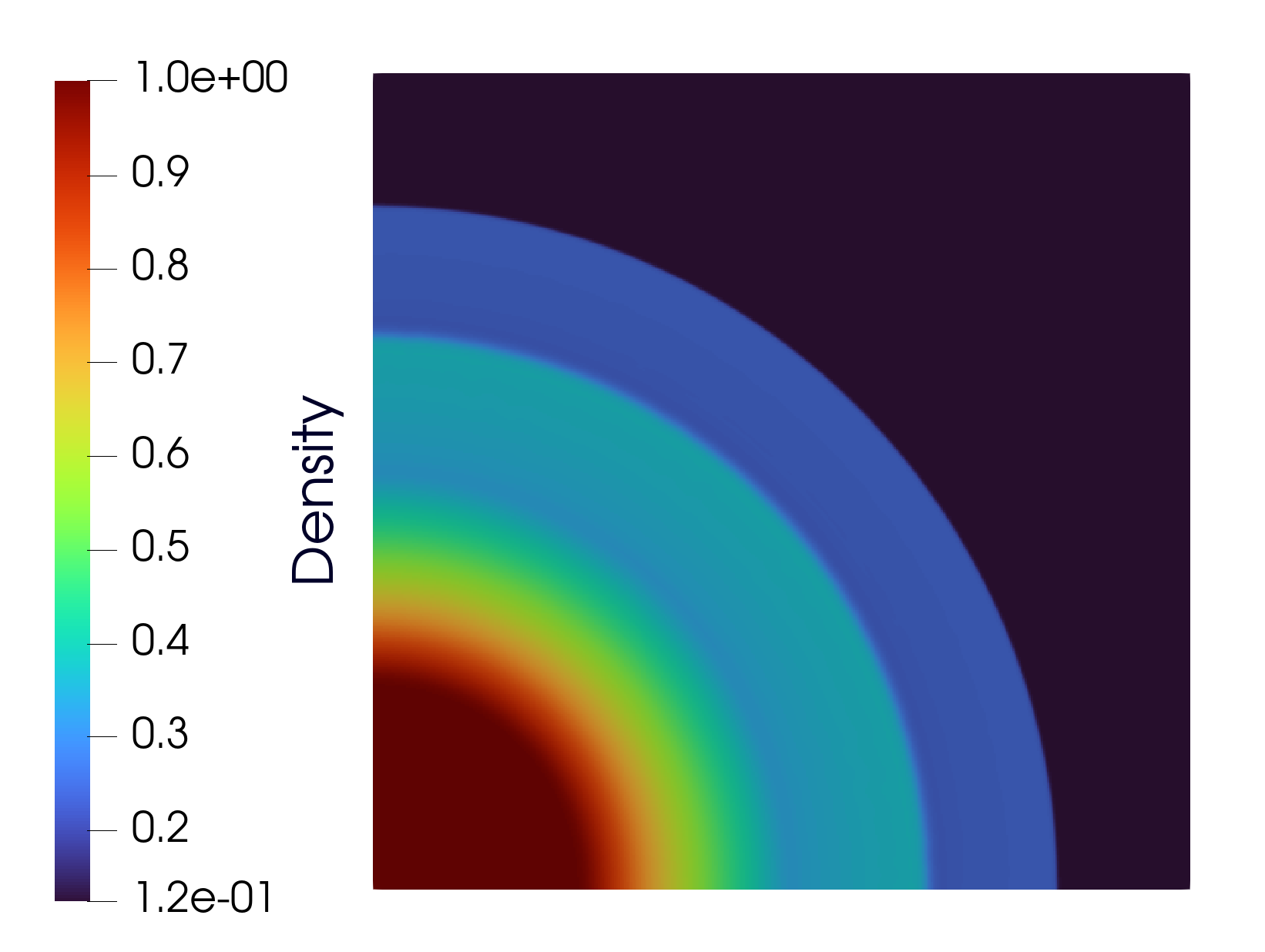}&
\includegraphics[width=0.5\textwidth]{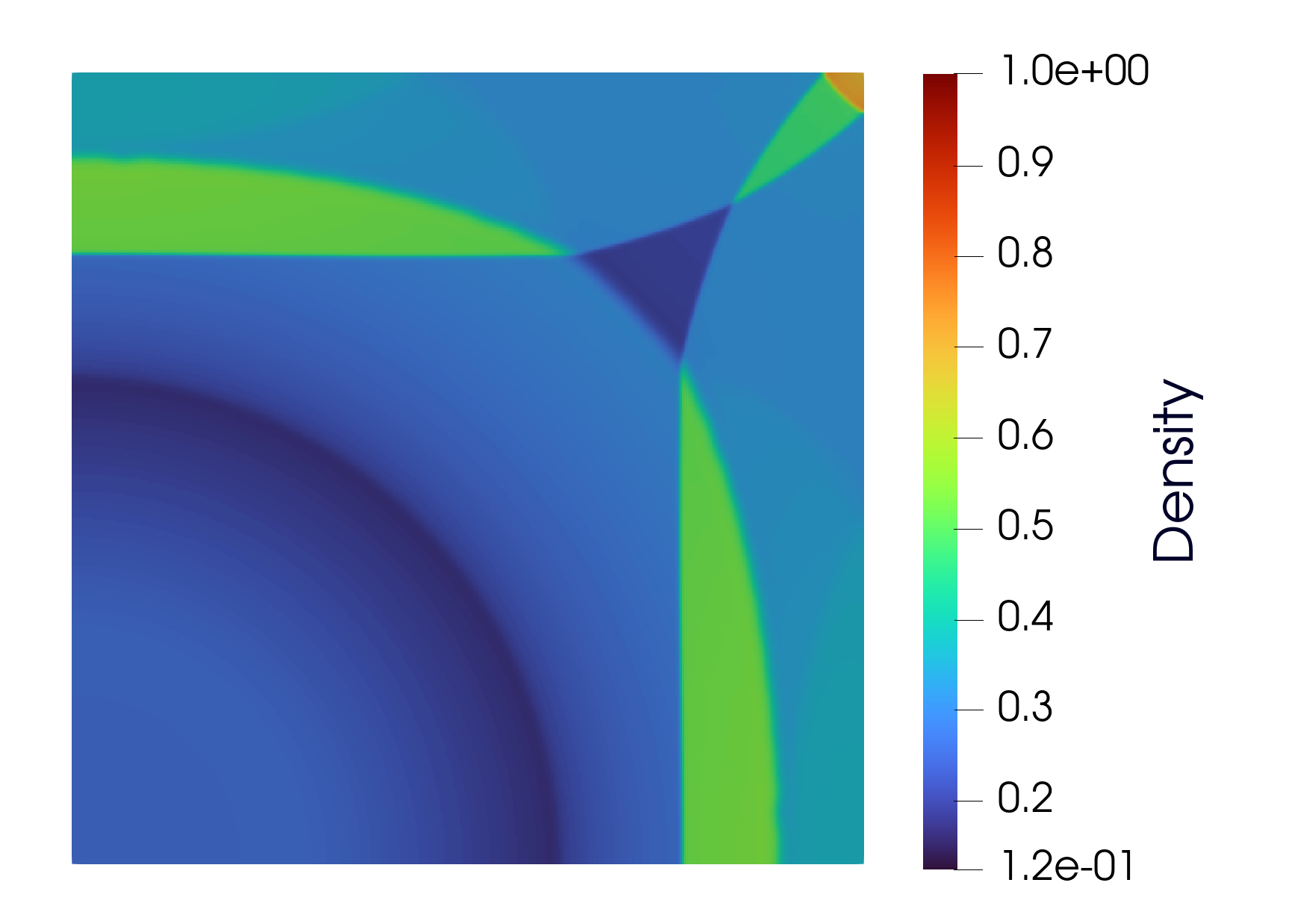} \\
\includegraphics[width=0.47\textwidth]{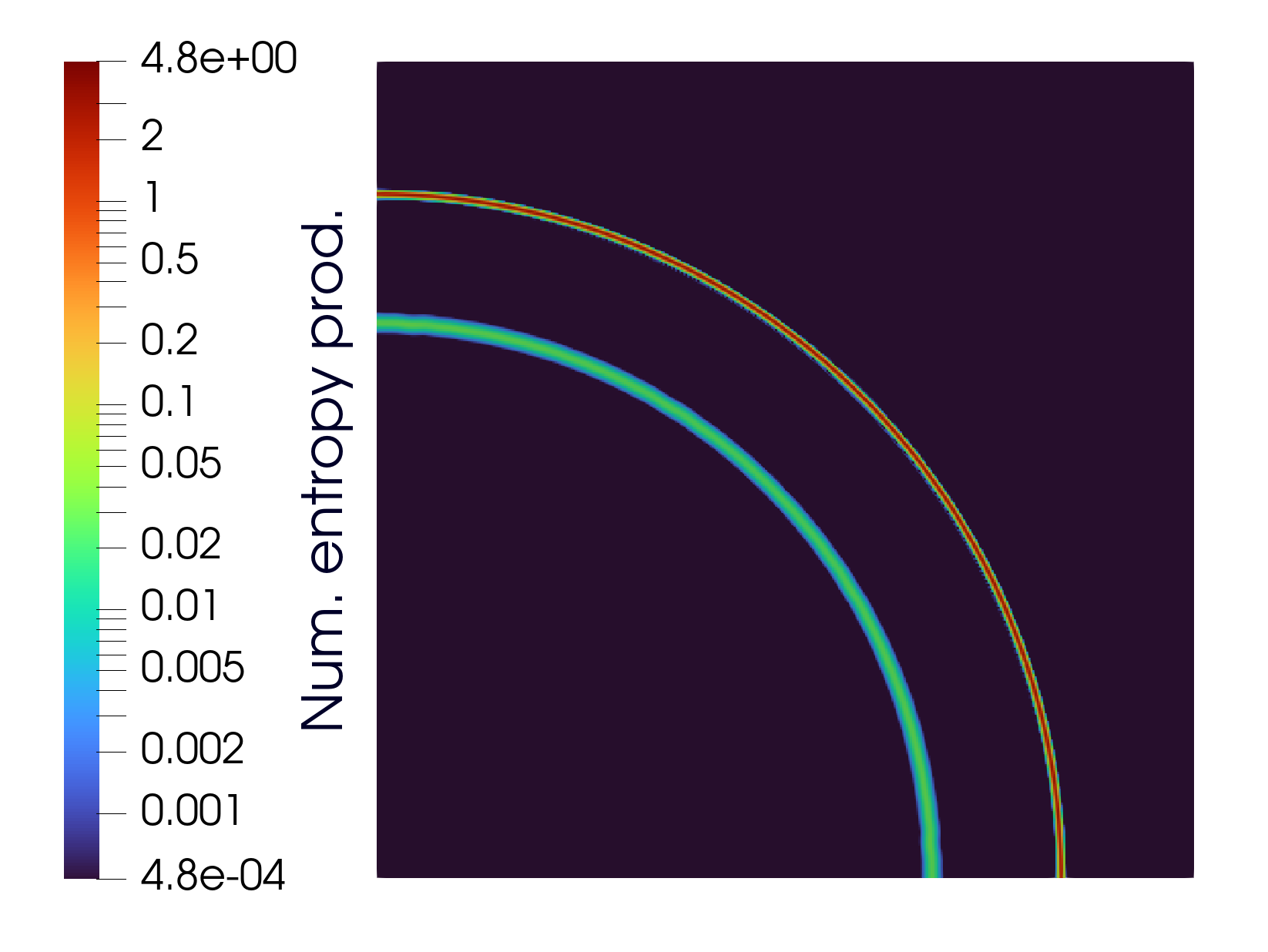}& 
\includegraphics[width=0.5\textwidth]{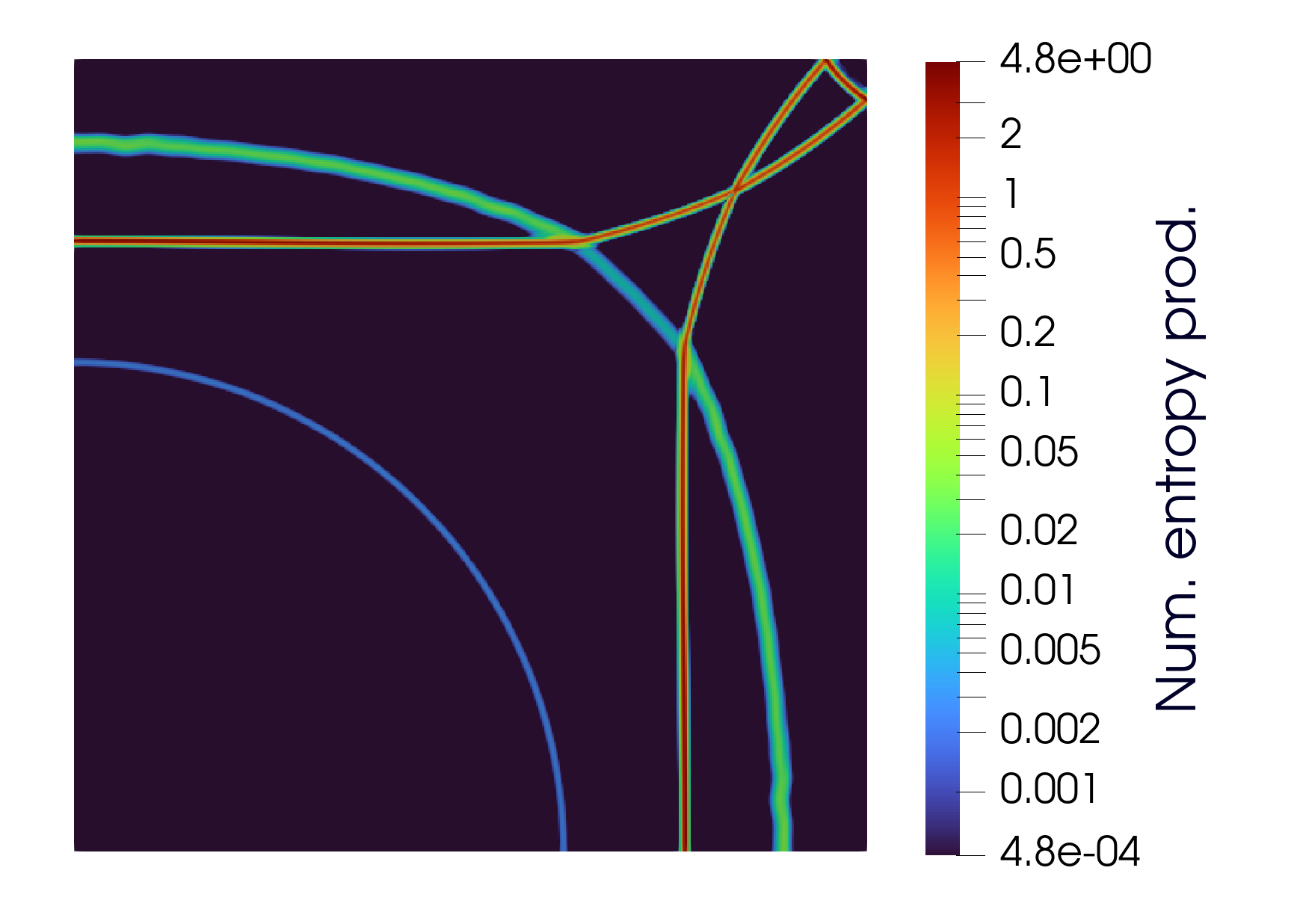}
\end{tabular}
\caption{Radial Sod problem: density and corresponding numerical entropy production (in logarithmic scale) computed with $P_0P_2$ at $t=0.2$ (left panels) and at $t=0.6$ (right panels).}
\label{sod}
\end{figure}

\paragraph{Shock-Bubble interaction}
Next we propose the shock-bubble interaction problem \cite{CT:09}, with initial data 
\[
\begin{aligned}
   & [\rho,u,v,p]=[11/3,2.7136021011998722,0,10] &\mbox{ if } x<0,\\
   & [\rho,u,v,p]=[1,0,0,1] & \mbox{ if } \sqrt{((x-0.3)^2+y^2)}>0.2, 
\end{aligned}
\]
and $[\rho,u,v,p]=[0.1,0,0,1]$ elsewhere. We take as domain $[-0.1,1.6]\times[-0.5,0.5]$ and as computational domain $[-0.1,1.6]\times[0,0.5]$ because of the symmetry in $y$. This initial data gives rise to a shock wave that hits a stationary bubble of gas at low pressure. We impose Dirichlet boundary conditions on the left, free-flow on the right, symmetry at $y=0$ and solid walls at $y=0.5$. 

When the shock hits the bubble, it gives rise to a complex unstable configuration, that leads to different final shapes using different schemes or different resolutions. 
In Fig.~\ref{fig10} and Fig.~\ref{fig11} we show the result at $t=0.15$ and $t=0.4$, after the interaction between the shock and the bubble, using a $P_0P_2$ scheme and a grid of $850\times250$ cells. Despite the complexity of the solution, notice that $S_j^n$ again detects the cells in which the solution exhibit some kind of discontinuity.
\begin{figure}[h]
	\centering
	\includegraphics[width=0.7\textwidth]{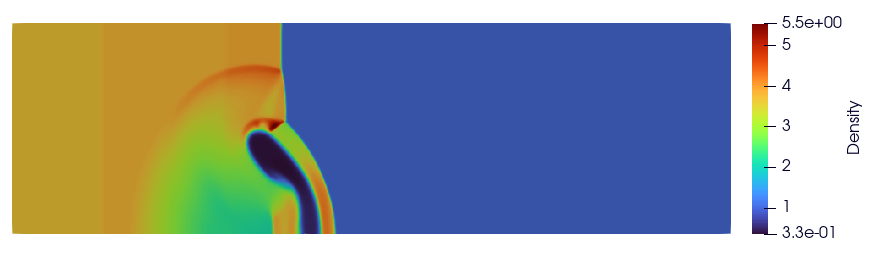}
        \includegraphics[width=0.7\textwidth]{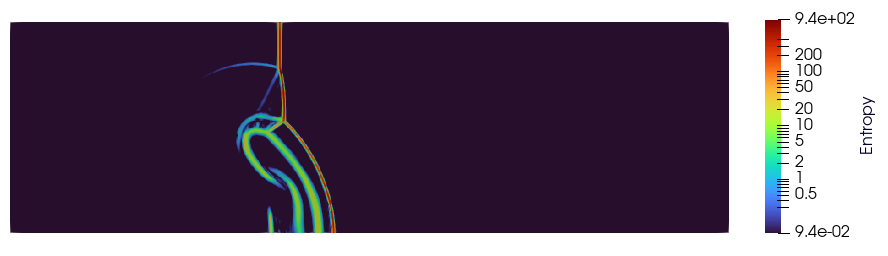}
\caption{Shock-bubble interaction problem: density and numerical entropy production (in logarithmic scale) at time $t=0.15$.}
\label{fig10}
\end{figure}
\begin{figure}[h]
	\centering
        \includegraphics[width=0.7\textwidth]{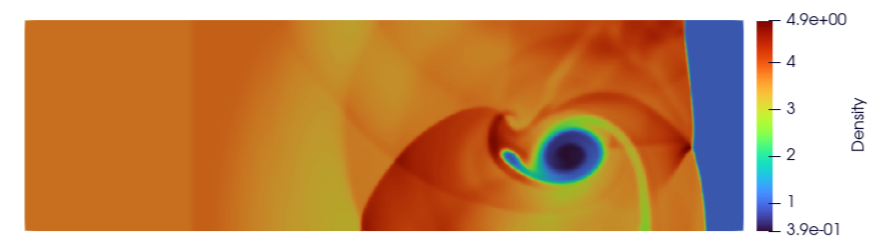}
        \includegraphics[width=0.7\textwidth]{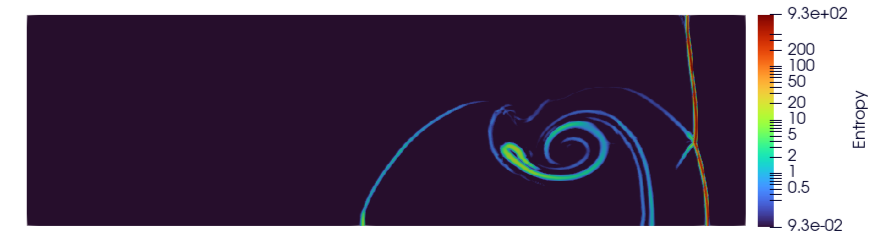}
\caption{Shock-bubble interaction problem: density and numerical entropy production (in logarithmic scale) at time $t=0.4$.}
\label{fig11}
\end{figure}

\section{Adaptivity}

Since the numerical entropy production scales as the local truncation error of the method and it detects the presence of discontinuities, it can be exploited as an indicator to drive the adaptation in adaptive schemes.

When dealing with high-order numerical methods, it is possible to observe a reduction of the real accuracy and the presence of spurious oscillations near discontinuities. In this section, we describe an example of adaptive scheme, in which $S_j^n$ is used as an indicator to eliminate those spurious oscillations, locally reducing the order of accuracy.

The idea is to fix a reference value $S_{ref}$ for $S_j^n$ and, whenever $|S_j^n|\geq S_{ref}$, in the corresponding cell the solution is recomputed using a lower order predictor. The value $S_{ref}$ is chosen such that those cells in which the solution is smooth or does not present spurious oscillations are not recomputed.

The scheme works as follows. 

\begin{minipage}{\textwidth}
\vspace{0.5cm}
\hrule

\vspace{0.1cm}
\textbf{Algorithm}

\vspace{0.1cm}
\hrule
\begin{enumerate}
    \item Apply a time step of the FV-ADER scheme $\eqref{fv5}$.
    \item Compute the numerical entropy production $S_j^n$ with the formula $\eqref{sader}$.
    \item Mark the cells in which $|S_j^n|\geq S_{ref}$.
        \begin{itemize}
            \item In the marked cells, recompute the ADER predictor with a lower order of accuracy, 
            \item apply \eqref{fv5} to update the numerical solution in the marked cells and their first neighbours.
        \end{itemize}
\end{enumerate}    
\vspace{0.1cm}
\hrule
\end{minipage}

\vspace{0.5cm}
Notice that in order to have a conservative scheme, we need to recompute the numerical solution also in the neighbors of the marked cells. In these cells the numerical fluxes are computed using both the high-order and the low-order predictor. However, this does not affect significantly the computational cost of the method.

The reference value $S_{ref}$ is of course problem-dependent. Since the scheme reduces the order of accuracy, the area in which we recompute the solution should be precisely localized. Indeed, we want to lose as little information as possible, keeping a high-order scheme if the solution is smooth or if there are not too many oscillations. Therefore, it is not convenient to fix a-priori a reference value, but it should be chosen with respect to the problem, looking at the cells in which one wants to intervene. A general discussion on how to choose $S_{ref}$ can be found in \cite{SL:18:AMRMOOD}, where the problem of choosing proper refinement thresholds for AMR is discussed. In particular, $S_{ref}$ could be computed running first the simulation on a coarse mesh, which is usually very cheap. Then, exploiting the scaling behaviour of $S_j^n$ under different flow regularities, the expected values of $S_j^n$ on the fine grid for the real numerical computation can be computed and used to set $S_{ref}$ according to the features of the solution of interest.

If we are dealing with particularly strong initial data, it may happen that, when computing the ADER predictor at the beginning of each time step, high-order reconstructions generate spurious oscillations computing inconsistent data, for example negative pressure or density in Euler equations of gasdynamics. In this case, we introduce adaptation also at that level: in those cells in which a negative pressure has been generated, we recompute the reconstruction using lower degree polynomials. This idea is reminiscent of the  Physical Admissibility Detection criteria (PAD) of the MOOD approach \cite{mood} and is also employed to trigger the subcell finite volume limiter for ADER-DG schemes \cite{mood:subcell}.

\begin{minipage}{\textwidth}
\vspace{1cm}
\hrule

\vspace{0.1cm}
\textbf{Algorithm}

\vspace{0.1cm}
\hrule
\begin{enumerate}
    \item Apply a time step of the FV-ADER scheme $\eqref{fv5}$:
    \begin{enumerate}
        \item Apply high-order reconstruction. (e.g. CWENO)
        \item Compute ADER predictor:
        \begin{itemize}
            \item If the high-order reconstruction does not generate unphysical values (e.g. negative pressure), compute high-order predictor.
            \item If the high-order reconstruction has generated unphysical values, compute the predictor using a lower-order reconstruction.
        \end{itemize}
        \item Compute numerical fluxes.
        \item Update the solution.
     \end{enumerate}
    \item Compute the numerical entropy production $S_j^n$ with the formula $\eqref{sader}$.
    \item Mark the cells in which $|S_j^n|\geq S_{ref}$.
        \begin{itemize}
            \item In the marked cells, recompute the ADER predictor with a lower order of accuracy, 
            \item apply \eqref{fv5} to update the numerical solution in the marked cells and their first neighbours.
        \end{itemize}
\end{enumerate}    
\vspace{0.1cm}
\hrule
\end{minipage}

\subsection{Adaptive scheme tests}

Consider again Euler equations of gasdynamics. We implement an adaptive $P_0P_2$ numerical scheme: the base scheme is the third order $P_0P_2$ tested previously and the adaptive procedure reduces its order to 1 locally in order to better control spurious oscillations. Depending on the problem, we fix a reference value $S_{ref}$ for $S_j^n$. This value is chosen running the simulation firstly on a coarse mesh. In those cells in which $|S_j^n|\geq S_{ref}$, the solution is recomputed using a $P_0P_0$ scheme.

\paragraph{Implosion}
We consider the problem of a diamond-shaped converging shock with initial data $[\rho,u,v,p]=[1,0,0,1]$ if $x+y>0.15$ and $[\rho,u,v,p]=[0.125,0,0,0.14]$ elsewhere \cite{HLL:99}. For symmetry, we restrict the computational domain to $[0,0.3]^2$. 

We compute the numerical solution with a grid of $800\times800$ cells. We use the CWZb3 reconstruction \cite{STP23:cweno:boundary}. The initial discontinuity gives rise to a shock and a contact discontinuity, both moving towards the origin, and to a rarefaction wave that moves outwards. In Fig.~\ref{impl} we show the solution at time $t=0.03$, $t=0.06$ and $t=0.09$ {\revDue with the corresponding numerical entropy production, which also in this case highlights the presence and the kind of discontinuities}.

Plotting the profile of the density computed with a non-adaptive $P_0P_2$ scheme on a grid of $800\times800$ at time $t=0.03$, we notice some oscillations between the contact discontinuity and the rarefaction wave {\revDue (Fig.~\ref{impl1}, left panel)}. 
We recompute the solution using the proposed adaptive scheme, fixing $S_{ref}=0.01$ and $S_{ref}=0.1$. We compare it with the numerical solution computed with non-adaptive $P_0P_2$ and $P_0P_0$ schemes {\revDue (Fig.~\ref{impl1}, right panel)}. We observe that choosing an appropriate $S_{ref}$, we manage to reduce the oscillations produced by the non-adaptive $P_0P_2$ scheme without loosing too much resolution. From these results, a good value for $S_{ref}$ in this case is $0.1$, since it computes the correct intermediate state, still avoiding oscillations.

\begin{figure}[h]
\centering
\begin{tabular}{ccc}
\includegraphics[width=0.32\textwidth]{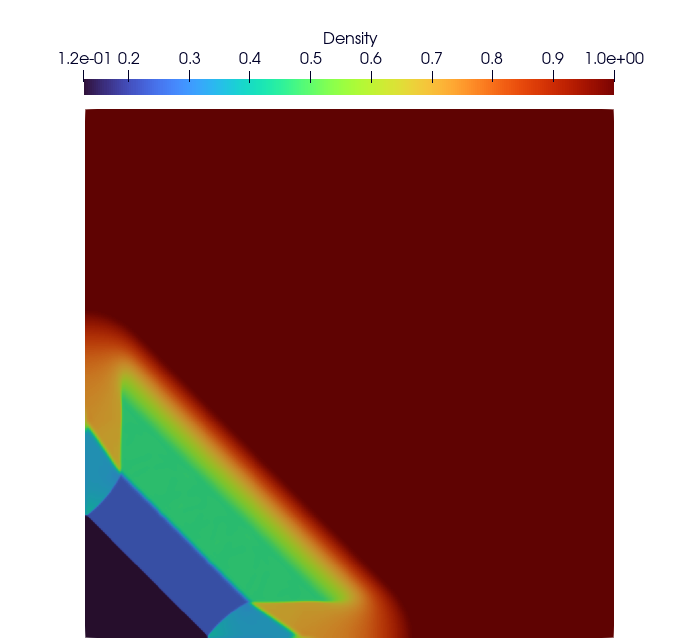} &
\includegraphics[width=0.32\textwidth]{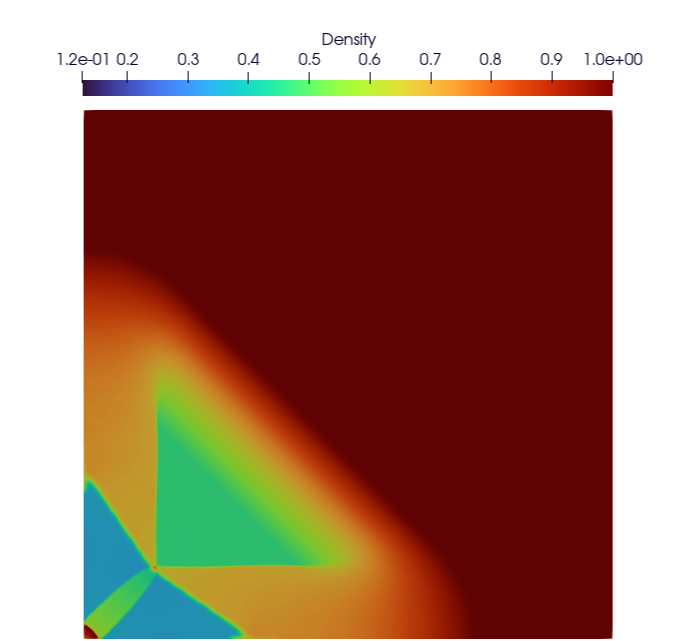} &
\includegraphics[width=0.32\textwidth]{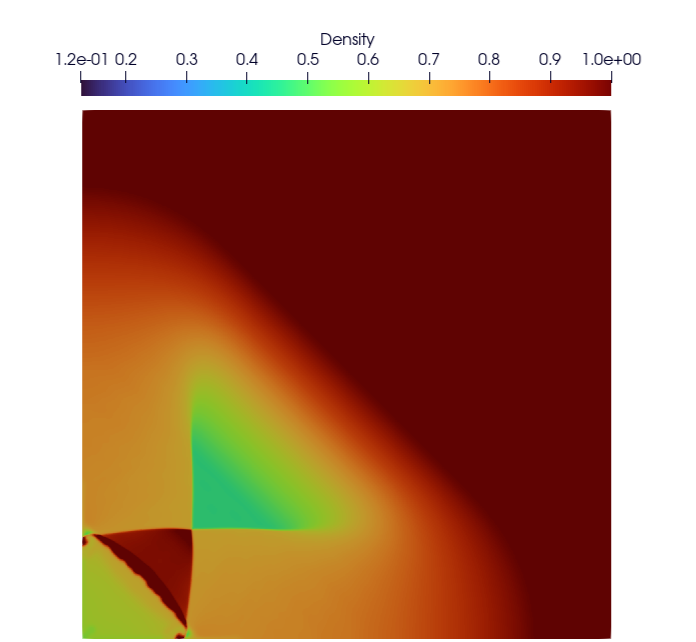} \\
\includegraphics[width=0.32\textwidth]{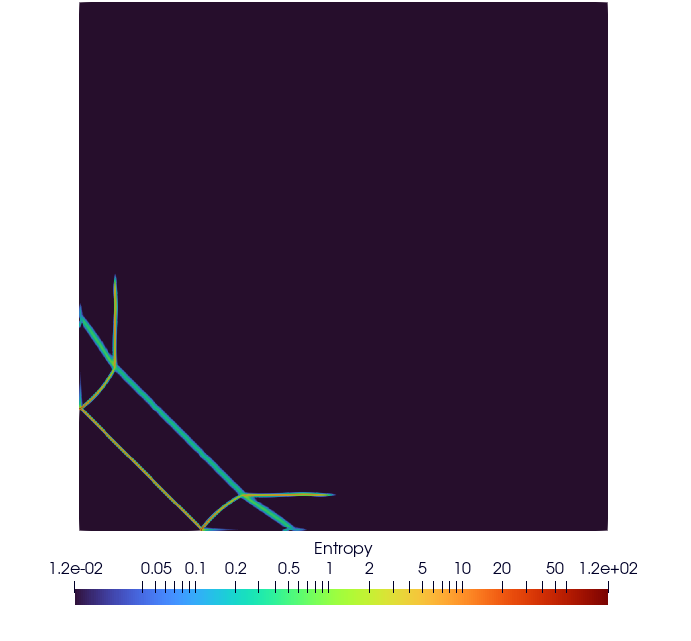} &
\includegraphics[width=0.32\textwidth]{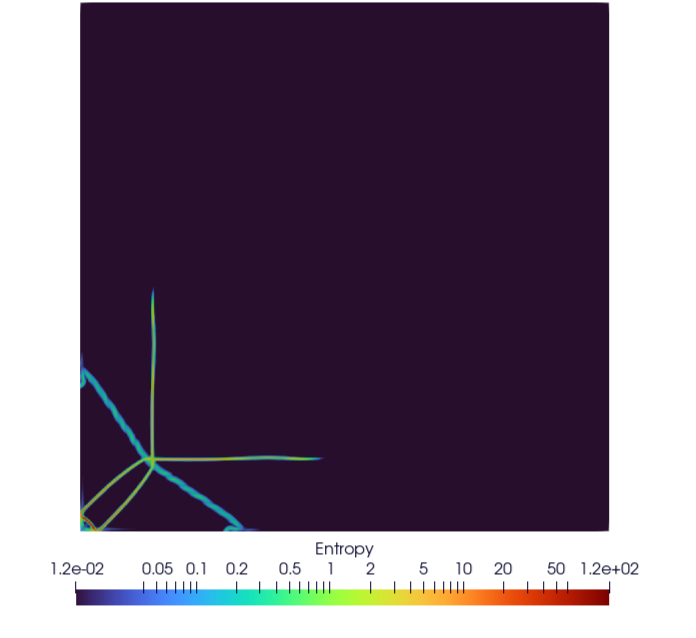} &
\includegraphics[width=0.32\textwidth]{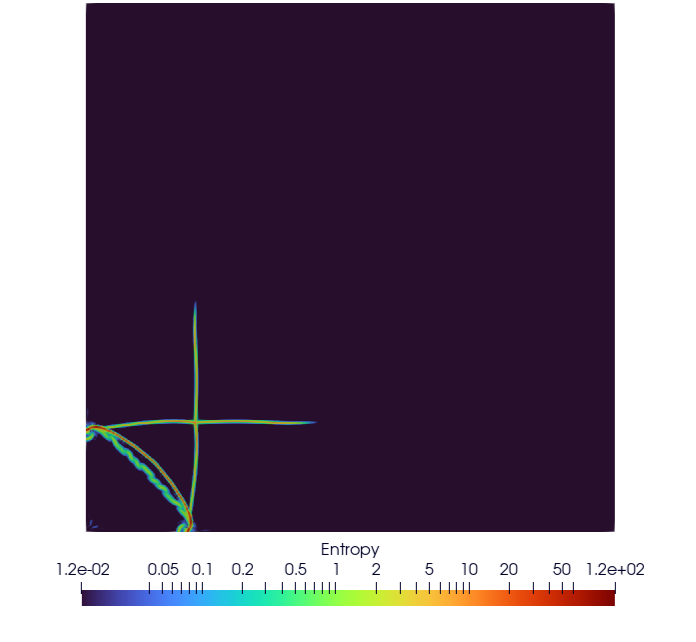}
\end{tabular}
\caption{Implosion problem: density and numerical entropy production (in logarithmic scale) at time $t=0.03$, $t=0.06$ and $t=0.09$.}
\label{impl}
\end{figure}

\begin{figure}[h]
    \centering
    \includegraphics[width=0.48\textwidth]{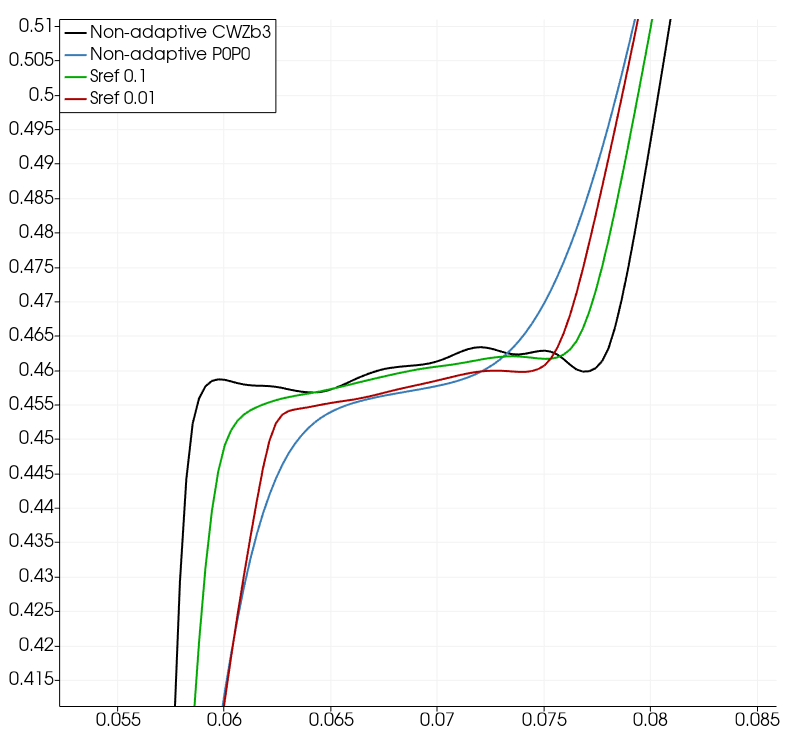}
    \includegraphics[width=0.48\textwidth]{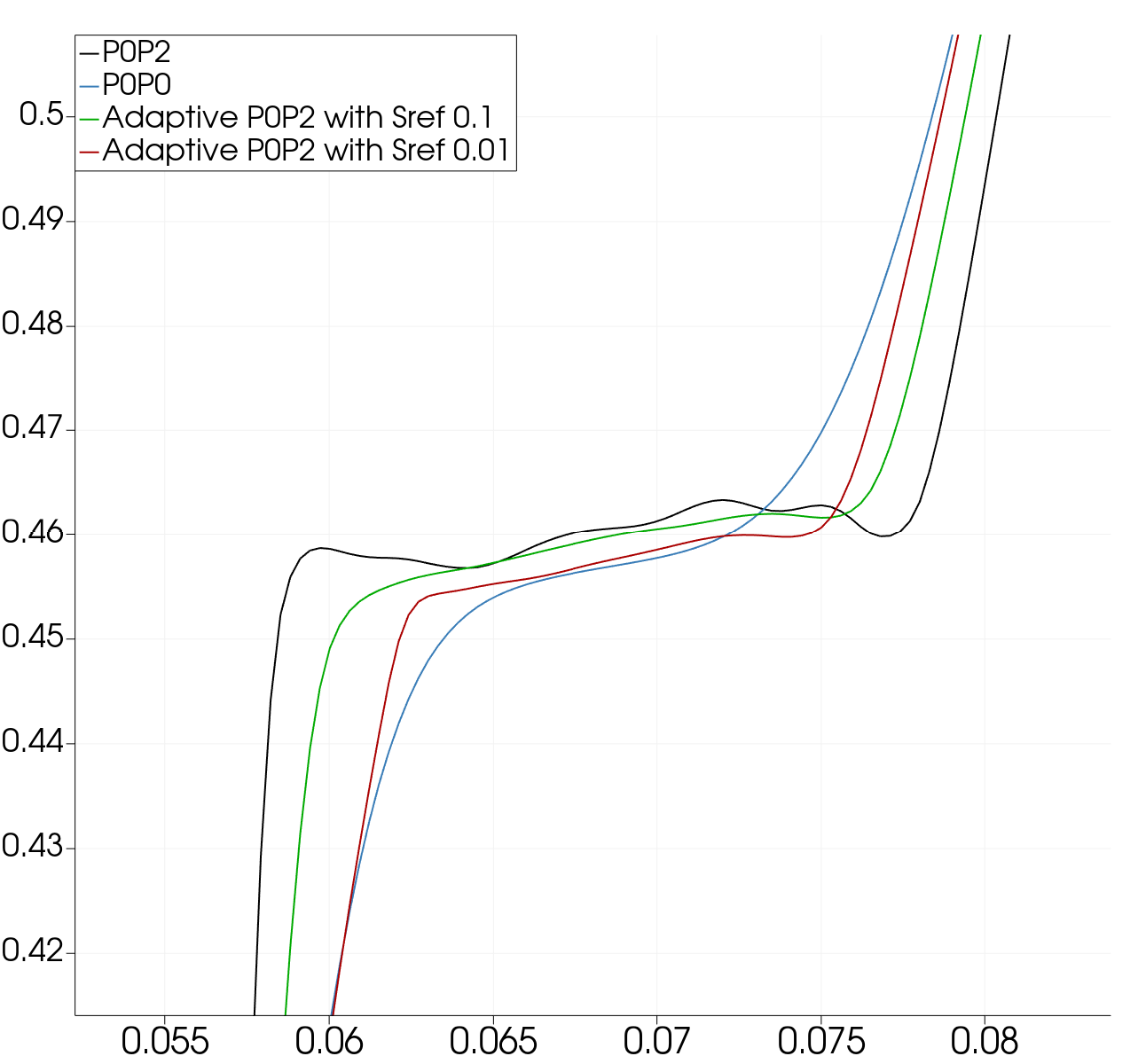}
\caption{Implosion problem. On the left: solution with $P_0P_2$ scheme at time $t=0.03$ on a grid of $800\times800$. On the right: zoom on the step between contact discontinuity and rarefaction wave computed using $P_0P_2$ and $p$-adaptive $P_0P_2$ schemes with $S_{ref}$ 0.1 and 0.01.}
\label{impl1}
\end{figure}

\paragraph{Woodward Colella test}
Next we present a test for which the code would break using a non-adaptive $P_0P_2$ scheme. We refer to \cite{WC:84}. We evolve as initial condition a gas with $\rho(x)=1$, $v(x)=0$ and
\begin{equation*}
    p(x)=
    \begin{cases}
        1000 & \mbox{ if } x\leq0.1 \\
        0.01 & \mbox{ if } x\in(0.1,0.9) \\
        100 & \mbox{ if } x\geq0.9
    \end{cases}
\end{equation*}
on $\Omega=[0,1]$ with $t_{Max}=0.038$ and wall boundary conditions. 

Firstly, we show the density solution at time $t=0.01$
in the first panel of Fig.~\ref{wc3}. In the right part of the domain, the solution presents a strong shock and a strong rarefaction that are pushed in opposite directions by the high initial pressure jump at $x=0.9$. The rarefaction is being reflected from the wall, giving rise to an almost constant density and pressure region. The shock and the rarefaction waves are separated by two constant states and a strong contact discontinuity. On the left of the domain, the solution presents a similar but more evolved structure due to the higher initial pressure jump at $x=0.1$. The rarefaction has already been reflected and it is already interacting with the strong contact discontinuity at $x=0.3$.
 In Fig.~\ref{wc3} we compare the numerical solution computed with a non-adaptive $P_0P_2$ scheme with the one computed with the corresponding adaptive scheme, fixing $S_{ref}=1$ and using a grid of 9600 cells and observe that the local order reduction is employed in very few cells. 

When the strong shock and the strong contact discontinuity move closer, the non-adaptive $P_0P_2$ scheme breaks down, whereas the adaptive one does not. We show the numerical solution at time $t=0.028$ in the middle panel of Fig.~\ref{wc3}, just after the collision, and at the final time $t=0.038$ in the right panel, using the adaptive $P_0P_2$ scheme with $S_{ref}=1$. We compare the solutions with the one computed with the first order $P_0P_0$ scheme. Notice that the adaptive approach allows to reach a better resolution than the use of a scheme of first order of accuracy. Notice also that the number of cells in which the order of accuracy is decreased (red line) is very small, and that they are tightly located around the main discontinuities.

\begin{figure}[h]
    \centering
    \includegraphics[width=0.32\textwidth]{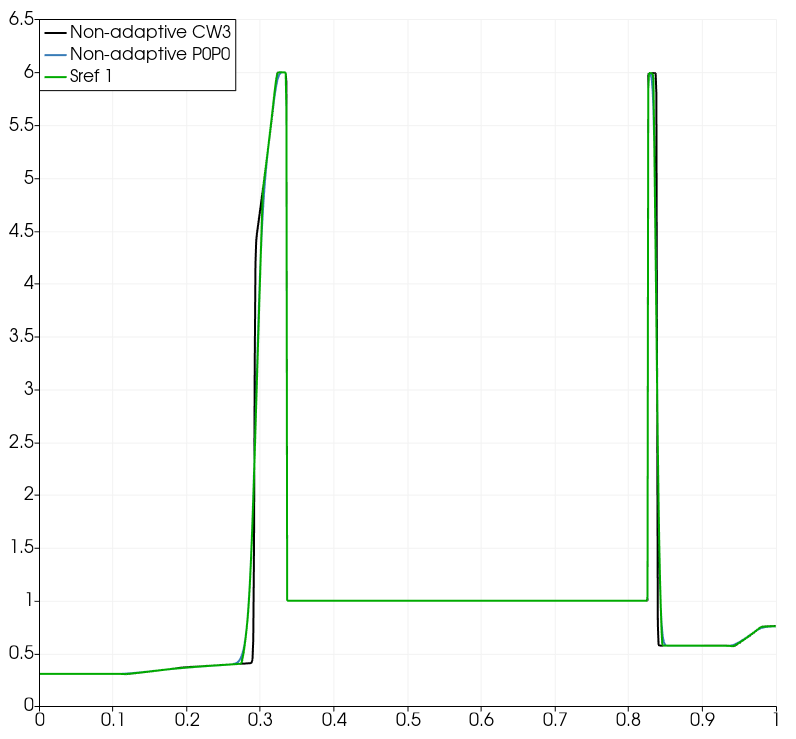}
    \includegraphics[width=0.32\textwidth]{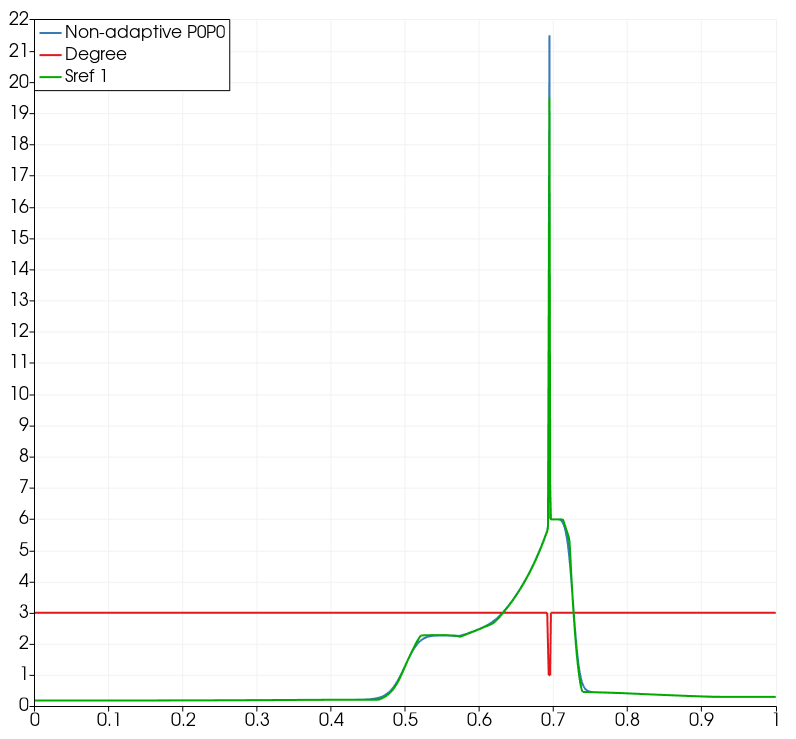}
    \includegraphics[width=0.32\textwidth]{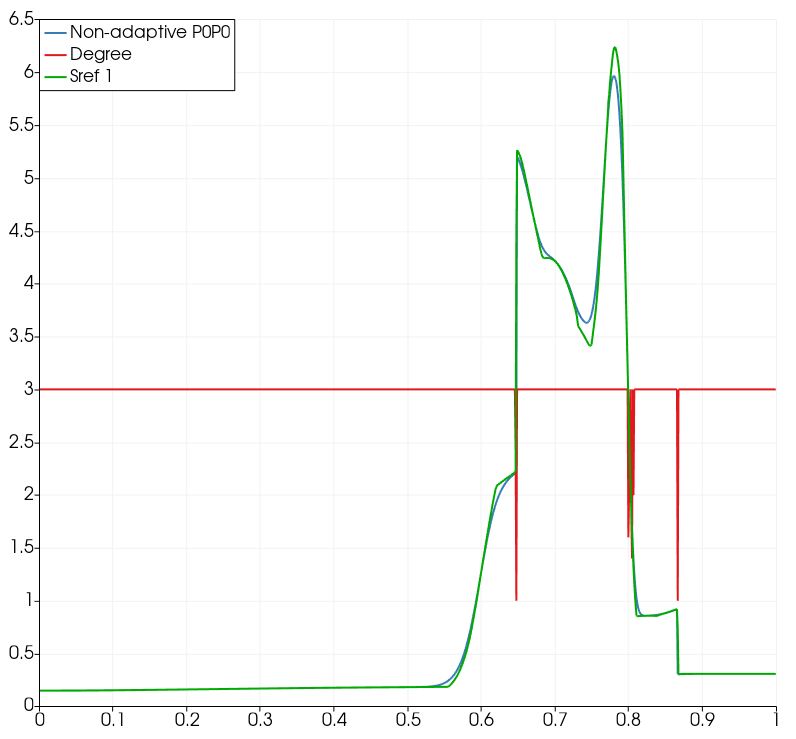}
\caption{ \revDue Woodward-Colella problem at time $t=0.01$, $t=0.028$ and $t=0.038$ using adaptive $P_0P_2$ scheme with $S_{ref}=1$ compared to the solution using $P_0P_0$. In the first panel, also the solution computed using $P_0P_2$ is shown, before it breaks down. The red line represents the order of accuracy of the adaptive scheme.}
\label{wc3}
\end{figure}

\paragraph{123-test problem}
Another challenging test is the 123-test problem from \cite{Toro:book}. The solution presents two strong rarefactions and a {\revDue trivial} stationary contact discontinuity. The initial condition is
\begin{equation*}
    \begin{cases}
        (\rho_L,v_L,p_L)=(1,-2,0.4) & \mbox{ if } x\leq0 \\
        (\rho_R,v_R,p_R)=(1, \text{ }\text{ }2,0.4) & \mbox{ if } x>0
    \end{cases}
\end{equation*}
in $[-0.5,0.5]$ with {\revUno $t_{Max}=0.15$}. We compute the numerical solution using the $p$-adaptive $P_0P_2$ scheme with $S_{ref}=200$ and a grid of 200 cells. At the beginning of the simulation (left panel of Fig.~\ref{fig123}), the order of the scheme is reduced in the cells in the center of the domain, in which the pressure is small. After the first timesteps, the simulation keeps the third order of accuracy until final time, which is shown in right panel of Fig.~\ref{fig123}. {\revUno We compare the quality of the final solution with the exact one.}
\begin{figure}
	\centering
        \includegraphics[width=1\textwidth]{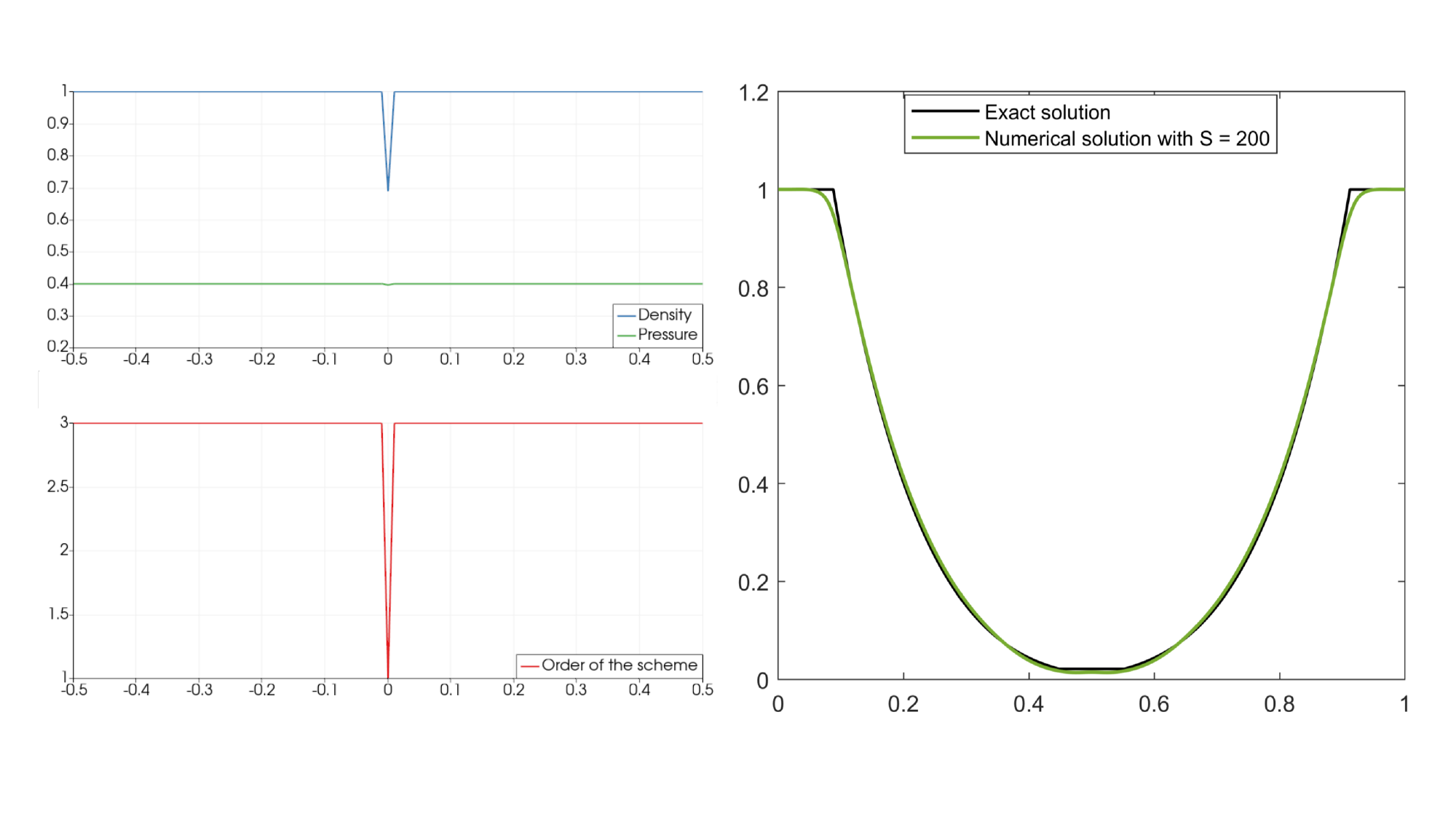}
\caption{123-test problem computed with $p$-adaptive $P_0P_2$ scheme fixing $S_{ref}=200$ with a grid of 200 cells.
{\revUno On the left: density, pressure and corresponding order of the scheme after the first timestep. In the following steps the scheme preserves order 3 in the whole domain. On the right: final solution at time $t_{Max}=0.15$.}}
\label{fig123}
\end{figure}

\paragraph{Sedov test}
A 2-dimensional test that breaks using a non-adaptive $P_0P_2$ scheme is the Sedov test \cite{Sedov:1959}. Its initial conditions are $[\rho,u,v,p,\gamma]=[1,0,0,10^{-6},\tfrac{7}{5}]$ over the domain $[-1.2,1.2]\times[-1.2,1.2]$. In the cells near the origin, the pressure value is 0.244816. This generates a point explosion of the gas. 
The non-adaptive $P_0P_2$ scheme breaks after a few time steps. One solution could be the use of a lower-order scheme, e.g. the $P_0P_0$. However, the adaptive approach allows to reach a better resolution in the cells in which the discontinuity is not present. In Fig.~\ref{sedov1} we show the solution computed with the adaptive $P_0P_2$ scheme on a grid of $400\times400$, fixing $S_{ref}=100$. The reference value has been chosen such that the number of cells in which the order is decreased is small enough not to compromise the accuracy of the solution in those regions in which it is not necessary. The red line detects the cells which have been computed with the lower-order scheme. Notice that the order is reduced only in those cells near the shock wave.

\begin{figure}[h!]
    \centering
    \includegraphics[width=0.48\textwidth]{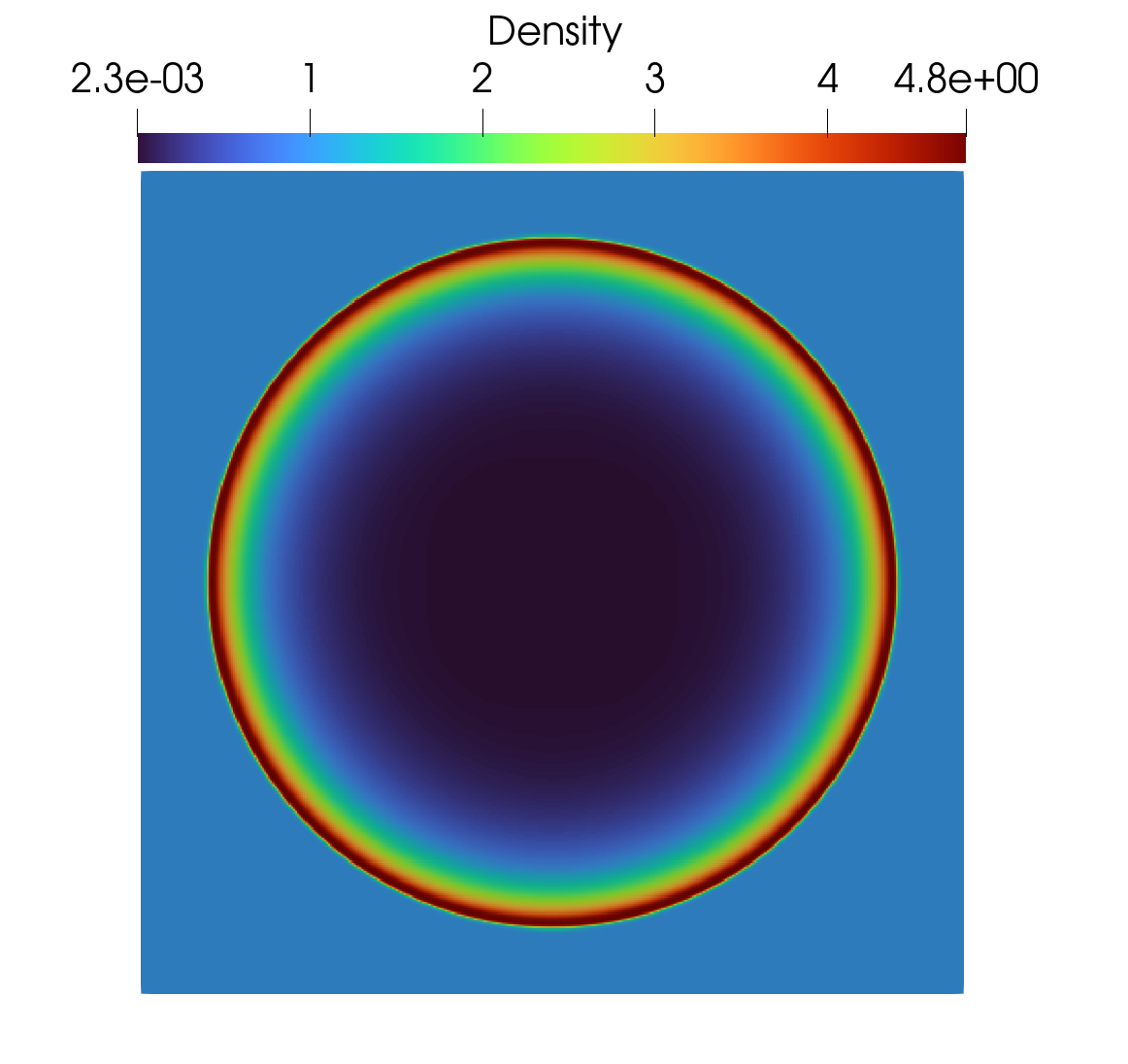}
    \hfill
    \includegraphics[width=0.48\textwidth]{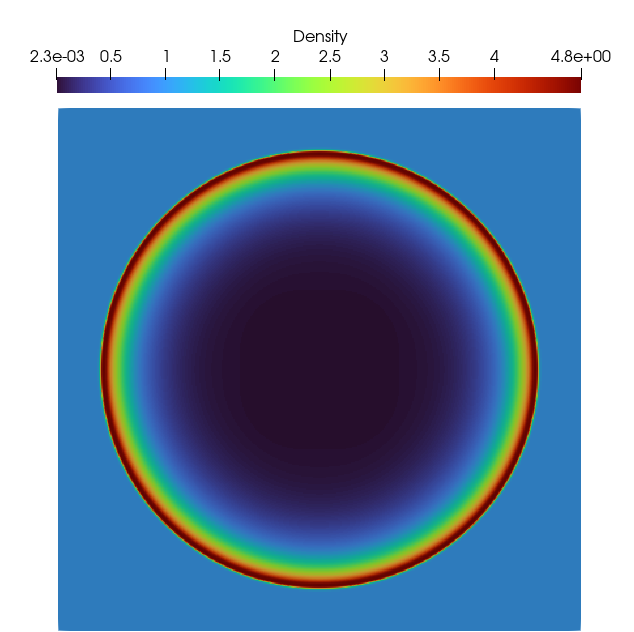}
\caption{Sedov problem computed with $p$-adaptive $P_0P_2$ scheme with $S_{ref}=100$ on a grid of $400\times400$ cells.}
\label{sedov1}
\end{figure}

\section{Conclusions}

In this work, we have proposed a definition for the numerical entropy production $S_j^n$ in the case of FV-ADER $P_0P_M$ schemes. We have proved theoretically, on multi-dimensional arbitrary grids, and verified numerically that at each time step $S_j^n$ converges to zero with the same rate of convergence of the numerical scheme in the case of smooth solutions and we have shown numerical evidence of the essentially negative sign of $S_j^n$ in the general case of a system of conservation laws. We have shown numerical evidence of the different behaviour of $S_j^n$ with respect to the local nature of the solution. On a general flow, $S_j^n$ is always bounded by terms of order $\mathcal{O}\left(1/\Delta t\right)$ and, when a shock crosses the space-time volume $[t^n,t^{n+1}]\times\Omega_j$, $S_j^n\sim C/\Delta t$. On smooth flows, including the inner part of a rarefaction wave, $S_j^n=\mathcal{O}(\Delta t^p)$ reflecting the local truncation error. On solutions with intermediate regularity, $S_j^n=\mathcal{O}(\Delta t^r)$ with $r=1$ on kinks and rarefaction corners and $r=0$ on contact discontinuities, again reflecting the local truncation error of the scheme.

Because of its properties, the numerical entropy production can be considered as a good error or smoothness indicator also in adaptive FV-ADER schemes. In this paper, we have proposed a simple example of $p$-adaptive scheme, as an illustration of a possible employment of this indicator in adaptive schemes. The proposed scheme modifies locally the order of accuracy with respect to the nature of the solution, guided by $S_j^n$. Indeed, near discontinuities high-order schemes may generate spurious oscillations, that can be reduced decreasing locally the order of accuracy. The numerical entropy production is a good a-posteriori indicator for this purpose. We intend to pursue the development and study of a more complete $hp$-adaptive scheme along the lines of the semidiscrete ones of \cite{SCR:cwenoAMR,SL:18:AMRMOOD} in a future work.

\section*{Acknowledgements}
This work was supported by the PRIN project ``High order structure-preserving semi-implicit schemes for hyperbolic equations'' (code 2022JH87B4).
Both authors are members of the GNCS–INDAM (National Group for Scientific Computing, Italy).

\bibliography{entfvader.bib}

\end{document}